\documentclass[a4paper,USenglish,cleveref, autoref, thm-restate]{compgeomtop}

\title{A Theory of Spherical Diagrams}

\author{Giovanni Viglietta}{Department of Computer Science and Engineering, University of Aizu, Japan \and \url{http://giovanniviglietta.com/} }{viglietta@gmail.com}{https://orcid.org/0000-0001-6145-4602}{}

\authorrunning{G. Viglietta}

\Copyright{Giovanni Viglietta}


\relatedversion{} 

\acknowledgements{The author is grateful to J.~O'Rourke, C.\,D.~T\'oth, J.~Urrutia, and M.~Yamashita for interesting discussions. The author also thanks the anonymous reviewers of CCCG~2022 and CGT for useful suggestions.}

\nolinenumbers 

\newcommand{\figscale}{0.57}

\Volume{42}
\ArticleNo{23}

\begin{document}

\maketitle

\begin{abstract}
We introduce an axiomatic theory of spherical diagrams as a tool to study certain combinatorial properties of polyhedra in $\mathbb R^3$, which are of central interest in the context of Art Gallery problems for polyhedra and other visibility-related problems in discrete and computational geometry.
\end{abstract}

\section{Introduction\label{sec:1}}
\paragraph*{Geometric intuition} Consider a set $\mathcal P$ of internally disjoint opaque polygons in $\mathbb R^3$ and a viewpoint $v\in \mathbb R^3$ such that no vertex of any polygon in $\mathcal P$ is visible to $v$. An example is given by the set of six rectangles in \cref{fig:intro} (left) with respect to the point $v$ located at the center of the arrangement.

Let $S$ be a sphere centered at $v$ that does not intersect any of the polygons in $\mathcal P$ (we may assume without loss of generality that $S$ is the unit sphere), and let $S_\mathcal P$ be the \emph{visibility map} of $\mathcal P$ with respect to $v$. That is, $S_\mathcal P$ is the set of radial projections onto $S$ of the portions of edges of polygons in $\mathcal P$ that are visible to $v$ (i.e., where polygons occlude projection rays). \cref{fig:intro} (right) shows an example of such a projection. The resulting structure $S_\mathcal P$ is a \emph{Spherical Occlusion Diagram}; this name indicates that no vertices of $\mathcal P$ appear in the visibility map, because they are all occluded by polygons.

In this paper we set out to formalize an axiomatic theory of Spherical Occlusion Diagrams and study their combinatorial structure.

\begin{figure}
  \centering
  \includegraphics[width=\linewidth]{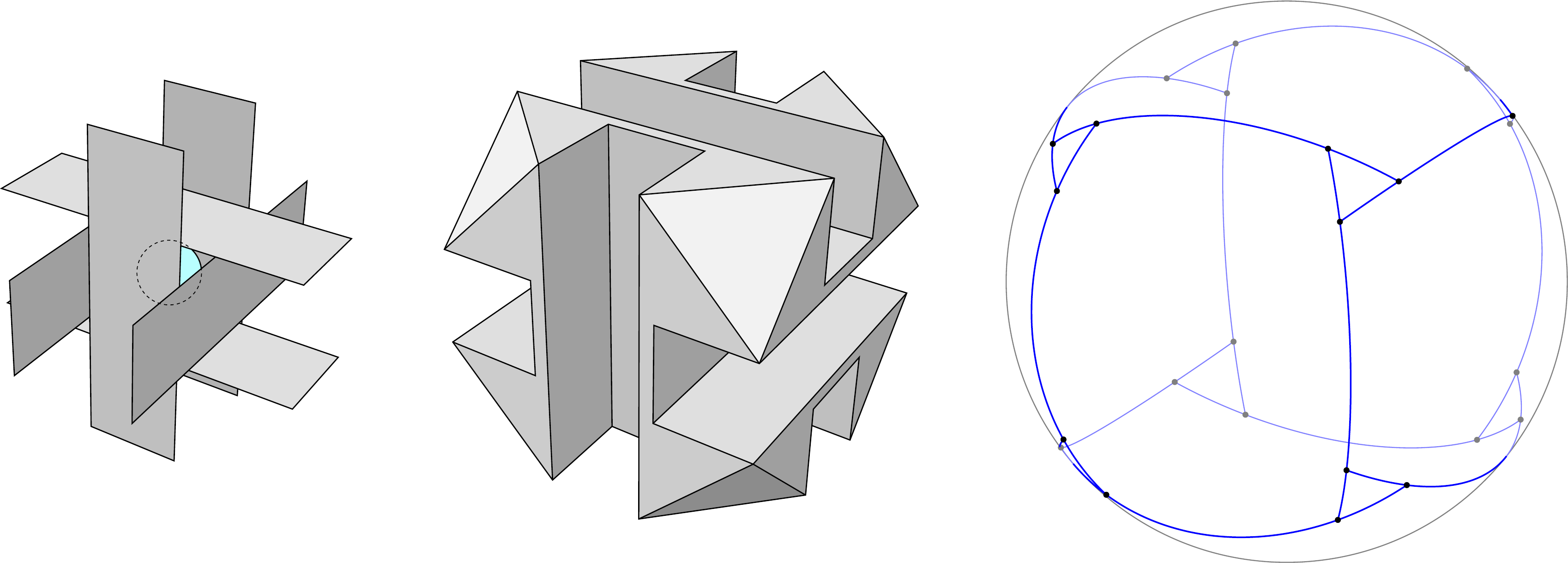}
  \caption{Construction of a Spherical Occlusion Diagram (right) as the visibility map of an arrangement of six rectangles (left) or the visibility map of a polyhedron with respect to the central point, which does not see any vertices (center).}
  \label{fig:intro}
\end{figure}

\paragraph*{Applications} Spherical Occlusion Diagrams naturally arise in visibility-related problems for arrangements of polygons in $\mathbb R^3$, and especially for polyhedra.

An example is found in~\cite{cano}, where an upper bound is given on the number of edge guards that solve the Art Gallery problem in a general polyhedron. That is, given a polyhedron  $\mathcal P$, the problem is to find a (small) set of edges that collectively see the whole interior of  $\mathcal P$. An edge $e$ sees a point $x$ if and only if there is a point $y\in e$ such that the open line segment $xy$ does not intersect the boundary of $\mathcal P$.\footnote{This definition of visibility is slightly more restrictive than the one most commonly adopted in the Art Gallery literature, which allows the segment $xy$ to graze the boundary of $\mathcal P$ without properly crossing it. However, for the purposes of this paper, the two definitions are essentially equivalent and lead to the same theory of Spherical Occlusion Diagrams. The choice of the less common definition is due to the fact that it prevents visibility maps of polyhedra from having multiple coincident arcs, as well as other pathological configurations. In turn, this allows for a slightly more elegant and concise formulation of \cref{p:1}, which relates visibility maps and Spherical Occlusion Diagrams.} (The reader can refer to~\cite{benbernou,reflex} for more results on this problem, as well as~\cite{art,orourke} for surveys on the Art Gallery problem in 2-dimensional settings.)

The idea of~\cite{cano} is to preliminarily select a (small) set $E$ of edges that cover all vertices of $\mathcal P$. Note that $E$ may be insufficient to guard the interior of $\mathcal P$, as some of its points may be invisible to all vertices; \cref{fig:intro} (center) shows an example. Thus, an additional (small) set of edges $E'$ is selected, which collectively see all internal points of $\mathcal P$ that do not see any vertices. Clearly, $E\cup E'$ is a set of edges that see all internal points of $\mathcal P$.

The selection of the set of edges $E'$ is carried out in~\cite{cano} by means of an ad-hoc analysis of some properties of points that do not see any vertices of $\mathcal P$. Spherical Occlusion Diagrams offer a systematic and general tool to reason about points in a polyhedron that do not see any vertices.

Spherical Occlusion Diagrams have also provided a framework for proving the main result of~\cite{tjc,arxiv}: Any point that sees no vertex of a polyhedron must see at least 8 of its edges, and the bound is tight.

\paragraph*{Summary} In \cref{sec:2} we give an axiomatic theory of Spherical Occlusion Diagrams and discuss their realizability as visibility maps. In \cref{sec:3} we prove some basic properties of Spherical Occlusion Diagrams, while in \cref{sec:4} we focus on an important pattern called \emph{swirl}. \cref{sec:5,sec:6} are devoted to two interesting classes of Spherical Occlusion Diagrams: the \emph{swirling} and the \emph{uniform} ones, respectively. \cref{sec:7} contains some concluding remarks and directions for future research.

A preliminary version of this paper appeared at~CCCG~2022~\cite{cccg}. In the present version, most sections have been reworked and all missing proofs have been included. In addition, \cref{p:07a,t:swirling1,t:swirling2} are new contributions.

\section{Axiomatic theory\label{sec:2}}
In the following, we will abstract from a specific set of polygons $\mathcal P$ and a viewpoint $v$, and we will focus on the salient properties of the Spherical Occlusion Diagrams constructed in \cref{sec:1} in order to devise a small set of axioms that describe all of them.

\paragraph*{Definitions}

Some terms will be useful. A \emph{great circle} on a sphere $S$ is a circle of maximum diameter within $S$. Equivalently, a great circle is the intersection between $S$ and a plane through its center. Two points on a sphere are \emph{antipodal} if the line through them contains the center of the sphere. A \emph{great semicircle} is an arc of a great circle whose endpoints are antipodal.

The \emph{geodesic arc} between two distinct non-antipodal points $x$ and $y$ on a sphere $S$ is the unique shortest path within $S$ having endpoints $x$ and $y$. Equivalently, it is the arc of the unique great circle through $x$ and $y$ that has $x$ and $y$ as endpoints and is strictly shorter than a great semicircle. Two geodesic arcs are \emph{collinear} if they lie on the same great circle.

\begin{definition}
Let $a$ and $b$ be two non-collinear geodesic arcs on a sphere. If an endpoint $p$ of $a$ lies in the relative interior of $b$, we say that $a$ \emph{hits} $b$ at $p$ (or \emph{feeds into} $b$ at $p$) and $b$ \emph{blocks} $a$ at $p$.
\end{definition}

\paragraph*{Axioms}
We are now ready to formulate an abstract theory of Spherical Occlusion Diagrams.

\begin{definition}\label{d:1}
A \emph{Spherical Occlusion Diagram (SOD)} is a finite non-empty collection $\mathcal D$ of geodesic arcs on the unit sphere in $\mathbb R^3$ satisfying the following axioms.
\begin{enumerate}
\item[A1.] Any two arcs in $\mathcal D$ are internally disjoint.
\item[A2.] Each arc in $\mathcal D$ is blocked by arcs of $\mathcal D$ at each endpoint.
\item[A3.] All arcs in $\mathcal D$ that hit the same arc of $\mathcal D$ reach it from the same side.
\end{enumerate}
\end{definition}

As an example, \cref{fig:example} shows an SOD with 18 arcs.

\begin{figure}[ht]
  \centering
  \includegraphics[scale=\figscale]{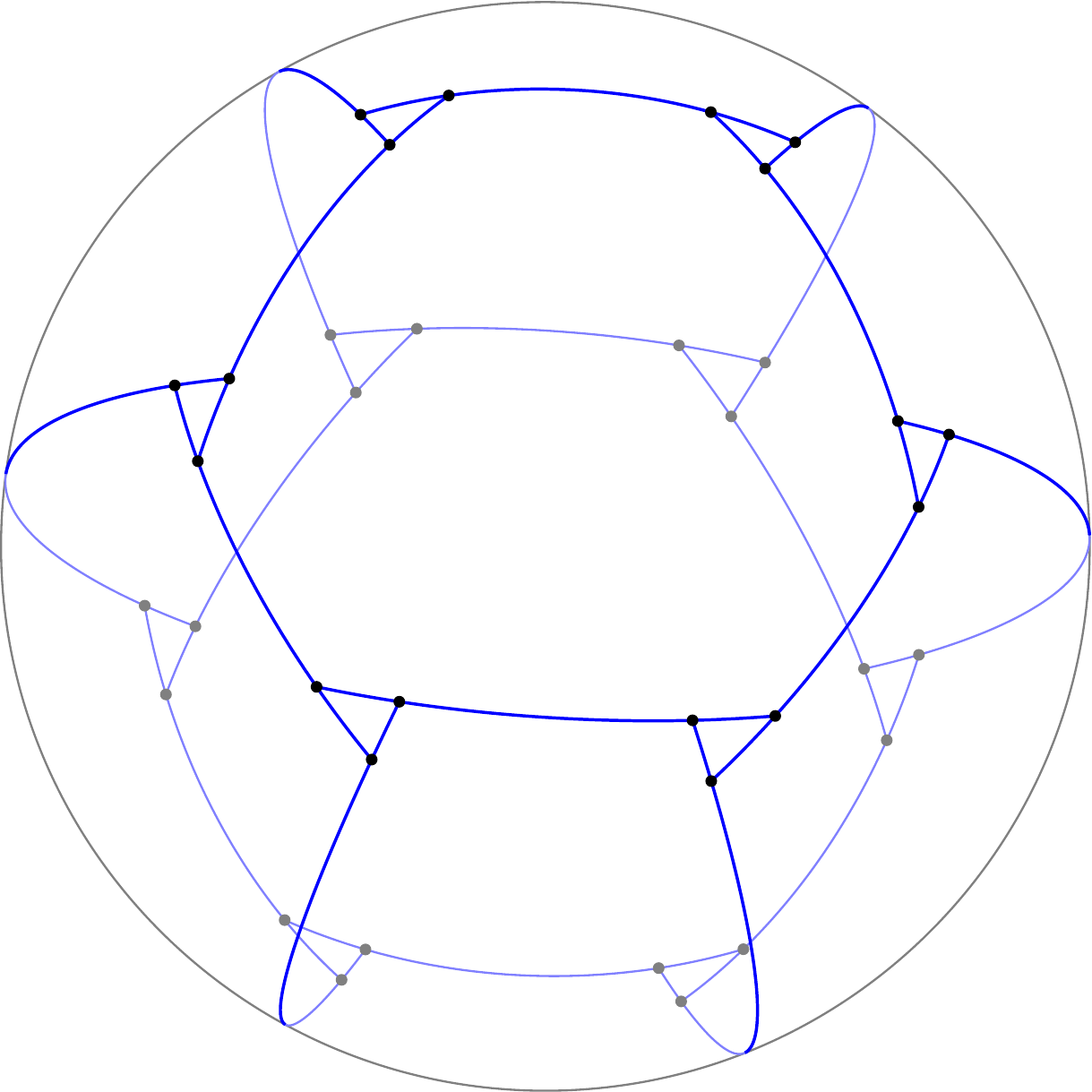}
  \caption{Example of an SOD with 18 arcs.}
  \label{fig:example}
\end{figure}

\paragraph*{Discussion}
We remark that the SOD axioms introduced in~\cite{cccg} are slightly different. In particular, the axiom~A1 therein states that, if two arcs $a,b\in\mathcal D$ have a non-empty intersection, then $a$ hits $b$ or $b$ hits $a$. Also, in~\cite{cccg} it is not postulated that arcs are shorter than great semicircles, but this property is derived as a theorem.

The axioms in the present paper are slightly more inclusive, in that they allow multiple arcs to share an endpoint, as long as the common endpoint is an interior point of some other arc. Allowing such ``degenerate configurations'' in SODs allows us to prove a slightly more general \cref{p:1}, whereas the one in~\cite{cccg} required the extra assumption that the viewpoint $v$ be in ``general position'' with respect to the polyhedron $\mathcal P$.

Moreover, if we removed the assumption that the arcs in an SOD are shorter than great semicircles, we would still be able to prove that arcs are \emph{not longer} than great semicircles, but this would not exclude some degenerate cases in which some arcs are great semicircles.

\paragraph*{Realizability} It is easy to recognize that the visibility maps $S_\mathcal P$ as constructed in \cref{sec:1} indeed provide a model for our theory, as they satisfy all its axioms, at least when $\mathcal P$ is a polyhedron.

\begin{proposition}\label{p:1}
The visibility map $S_\mathcal P$ of any polyhedron $\mathcal P$ with respect to any viewpoint $v$ that does not lie on the boundary of $\mathcal P$ and sees no vertices of $\mathcal P$ satisfies the axioms of Spherical Occlusion Diagrams.
\end{proposition}
\begin{proof}
$S_\mathcal P$ is obviously finite. To prove that it is not empty, consider a shortest path from $v$ to any vertex of $\mathcal P$; such a path must bend at the edges of $\mathcal P$, the first of which is visible to $v$.

For each arc $a\in S_\mathcal P$, let $e_a$ be the edge of $\mathcal P$ whose radial projection on the sphere (partly occluded by faces of $\mathcal P$) contains $a$. Since $e_a$ is a line segment of finite length, $a$ must be an arc of a great circle that is shorter than a great semicircle, i.e., a geodesic arc.

Also, since $e_a$ is an edge of a face $F\in\mathcal P$ that is partially visible to $v$, all arcs of $S_\mathcal P$ that touch the interior of $a$ must reach it from the same side (axiom~A3) and cannot continue past $a$ (axiom~A1), because such arcs correspond to edges partially hidden by $F$ (hence, $e_a$ must be a reflex edge of $\mathcal P$).

Finally, the fact that each vertex of $\mathcal P$ is occluded by some face translates into the property that each endpoint of each arc in $S_\mathcal P$ must lie in the interior of another arc of $S_\mathcal P$ (axiom~A2). Note that here we used the fact that $v$ does not lie on a face of $\mathcal P$.
\end{proof}

Note that \cref{p:1} also holds (with the same proof) for any finite non-empty arrangement $\mathcal P$ of internally disjoint polygons, none of which is coplanar with the viewpoint.

It is known that the converse of \cref{p:1} is not true, as not every SOD is the visibility map of a polyhedron~\cite{kimberly}. There is also compelling evidence that a stronger statement holds, which we state next.

\begin{definition}
An SOD $\mathcal D$ is \emph{irreducible} if no proper subset of $\mathcal D$ is an SOD.
\end{definition}

\begin{conjecture}\label{c:1}
There is an \emph{irreducible} Spherical Occlusion Diagram (satisfying the axioms in \cref{d:1}) that is not the visibility map $S_\mathcal P$ of any polyhedron $\mathcal P$ with respect to any viewpoint $v$ that sees no vertices of $\mathcal P$.\footnote{The counterexample in~\cite{kimberly} is not irreducible, because it is constructed by embedding a tiny non-realizable \emph{planar} diagram into a larger SOD, for example within the eye of a swirl (refer to \cref{sec:4} for a definition of ``swirl'').}
\end{conjecture}

We remark that \cref{c:1} automatically extends to arbitrary arrangements $\mathcal P$ of disjoint polygons; actually, it is sufficient to prove \cref{c:1} for \emph{simply connected} polyhedra. Indeed, a set of disjoint polygons $\mathcal P$ that gives rise to an SOD $\mathcal D$ with respect to a viewpoint $v$ can easily be augmented by adding a mesh of polygons whose edges are either shared with $\mathcal P$ or concealed from $v$ by polygons in $\mathcal P$. The resulting simply connected polyhedron gives rise to the same SOD $\mathcal D$ (a possible way of implementing this construction is found in~\cite{arxiv}).

\section{Elementary properties\label{sec:3}}
We will prove some basic properties of SODs. In fact, all theorems in this section, with the only exception of \cref{p:tileedge}, hold more generally for finite non-empty collections of geodesic arcs that satisfy axioms~A1 and~A2, but not necessarily~A3.

It is immediate to prove a stronger form of axiom~A2.

\begin{proposition}\label{p:04}
Every arc in an SOD hits exactly two distinct arcs, one at each endpoint.
\end{proposition}
\begin{proof}
Let $p$ be an endpoint of an arc $a$. By axiom~A2, $a$ hits at least one arc $b$ at $p$. If $a$ hit a second arc $b'$ at $p$, then $p$ would be interior to both $b$ and $b'$, contradicting axiom~A1.
\end{proof}

The following statement implies that no two arcs in an SOD can hit each other.

\begin{proposition}\label{p:05}
No two arcs in an SOD intersect in more than one point.
\end{proposition}
\begin{proof}
Two internally disjoint arcs sharing two points cannot be shorter than great semicircles, contradicting the assumption that an SOD consists of geodesic arcs (see \cref{fig:pr05}).
\end{proof}

\begin{figure}[ht]
  \centering
  \includegraphics[scale=\figscale]{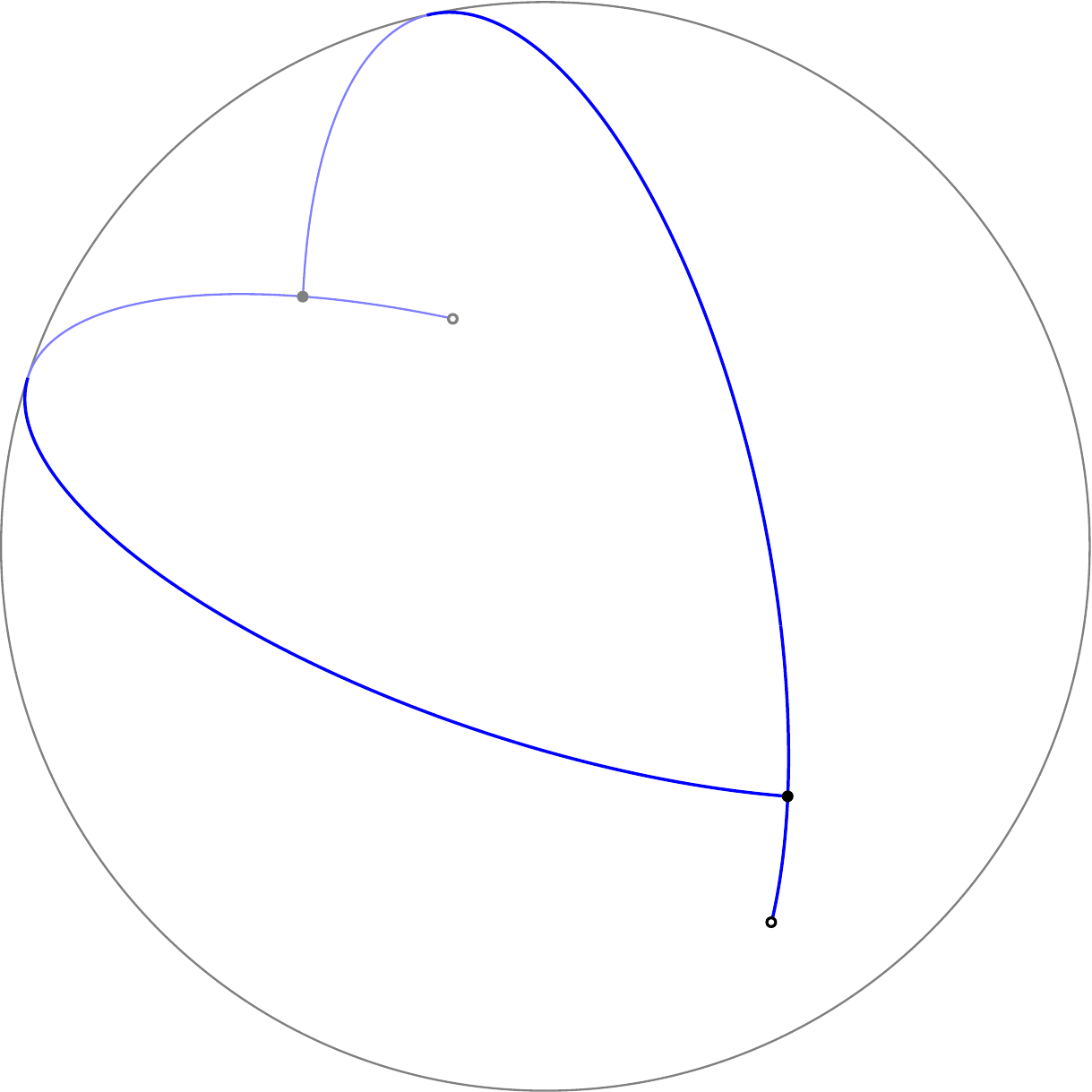}
  \caption{Two non-collinear arcs can only intersect in two antipodal points.}
  \label{fig:pr05}
\end{figure}

For the next results, we describe a general construction called ``monotonic walk''; refer to  \cref{fig:pr03}. Let $\mathcal D$ be an SOD, and let $p$ and $p'$ be antipodal points on the unit sphere. A \emph{clockwise monotonic walk around $p$} starting from a point $x_0\notin\{p,p'\}$ on an arc $a_0\in \mathcal D$ is a sequence of pairs $(x_i, a_i)$, where $x_i\in a_i \in\mathcal D$ for all $i\geq 0$, defined inductively as follows. Assuming $x_i$ is a point of $a_i\in\mathcal D$ distinct from $p$ and $p'$, we distinguish two cases. If the great circle containing $a_i$ also contains $p$ and $p'$, then $x_{i+1}$ is an endpoint of $a_i$ such that the geodesic arc $x_ix_{i+1}$ contains neither $p$ nor $p'$ (such an endpoint must exist, or else $a_i$ would not be shorter than a great semicircle; in the special case where $x_i$ is an endpoint of $a_i$, we take $x_{i+1}=x_i$). Otherwise, the great circle containing $a_i$ separates $p$ from $p'$. Following this great circle in the clockwise direction with respect to $p$ starting at $x_i$, we define $x_{i+1}$ as the first endpoint of $a_i$ encountered. In both cases, we let $a_{i+1}$ be the (unique, due to \cref{p:04}) arc of $\mathcal D$ that blocks $a_i$ at $x_{i+1}$.

\begin{figure}[ht]
  \centering
  \includegraphics[scale=\figscale]{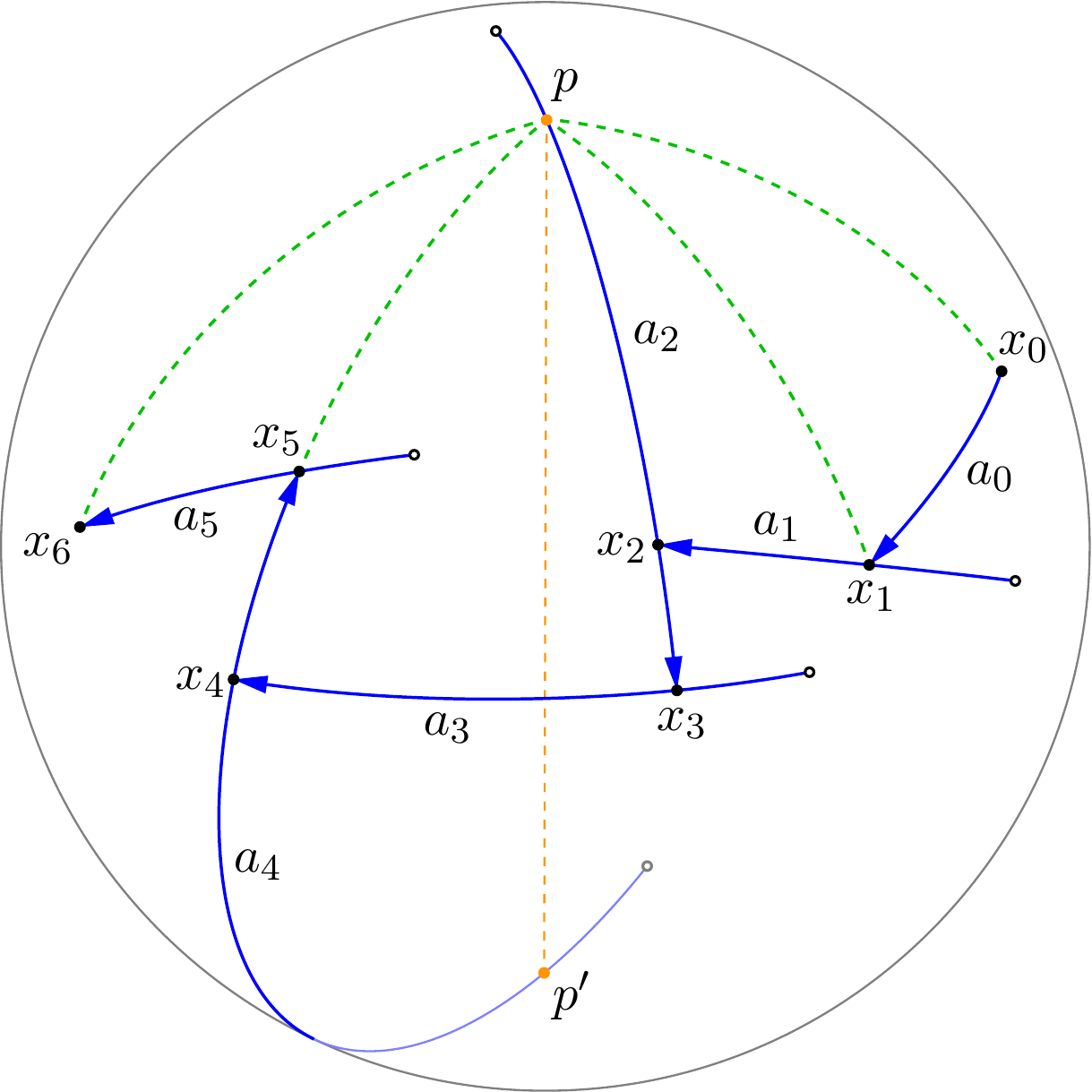}
  \caption{The initial steps of a clockwise monotonic walk around $p$ starting from $x_0\in a_0$.}
  \label{fig:pr03}
\end{figure}

A \emph{counterclockwise monotonic walk} around a point is defined similarly.

\begin{proposition}\label{p:semi}
Given an SOD $\mathcal D$, the relative interior of any great semicircle on the unit sphere intersects at least one arc of $\mathcal D$.
\end{proposition}
\begin{proof}
Let $p$ and $p'$ be the (antipodal) endpoints of a great semicircle $c$. Since $\mathcal D$ is not empty, there exists an arc $a_0\in\mathcal D$. Recall that the endpoints of a geodesic arc are distinct; therefore, $a_0$ contains infinitely many points, and in particular it contains a point $x_0$ distinct from $p$ and $p'$. Consider a clockwise monotonic walk $((x_i, a_i))_{i\geq 0}$ around $p$ starting from $x_0$. Observe that any $a_i$ that is an arc of a great circle through $p$ and $p'$ must be followed by an arc $a_{i+1}$ that is not. Hence, the monotonic walk makes steady progress around $p$. Since $\mathcal D$ is finite, in a finite amount of steps the monotonic walk touches the interior of every great semicircle with endpoints $p$ and $p'$, including $c$.
\end{proof}

\begin{proposition}\label{p:06}
An SOD partitions the unit sphere into spherically convex regions.
\end{proposition}
\begin{proof}
Let $\mathcal D$ be an SOD on the unit sphere $S$, let $D$ be the union of the arcs in $\mathcal D$, and let $p$ and $q$ be two points in the same connected component of $S\setminus D$. We will prove that $p$ and $q$ are connected by a single geodesic arc that does not intersect $D$.

By assumption and by the finiteness of $\mathcal D$, there is a chain $C$ consisting of $k$ geodesic arcs (drawn in green in \cref{fig:pr06}) that connects $p$ and $q$ without intersecting $D$. Let us choose $C$ so that $k$ is minimum. If $k=1$, there is nothing to prove. Thus, assume that $k\geq 2$.

Let $pp'$ and $p'p''$ be the first two geodesic arcs of $C$. If $p$ and $p''$ are antipodal, then $pp'$ and $p'p''$ are collinear, and their union is a great semicircle. By \cref{p:semi}, $D$ intersects $pp'\cup p'p''$, which is a contradiction. Hence $p$ and $p''$ are not antipodal, and there is a unique geodesic arc $pp''$ (drawn in orange).

Assume for a contradiction that $D$ intersects $pp''$ in $x\notin \{p,p''\}$. Then, either a clockwise or a counterclockwise monotonic walk around $p$ starting from $x$ must intersect $pp'\cup p'p''$, which is again a contradiction (see \cref{fig:pr06}). Therefore, $D$ does not intersect $pp''$, and we can replace $pp'$ and $p'p''$ in $C$ by the single geodesic arc $pp''$. Since this contradicts the minimality of $k$, we conclude that $k=1$.
\end{proof}

\begin{figure}[ht]
  \centering
  \includegraphics[scale=\figscale]{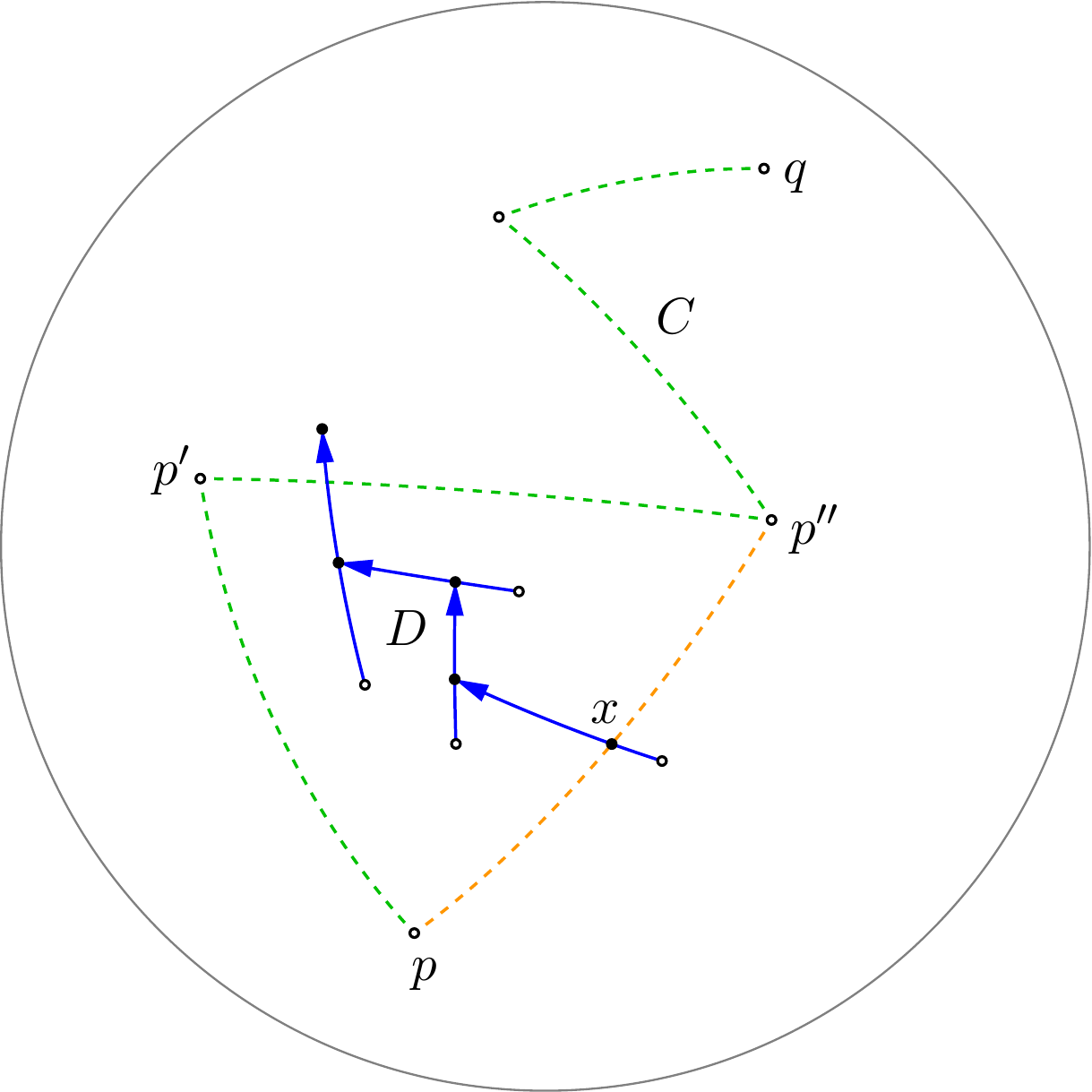}
  \caption{Proof of \cref{p:06}: The chain $C$ can be simplified by connecting $p$ and $p''$.}
  \label{fig:pr06}
\end{figure}

\begin{definition}
Each of the (spherically convex) regions into which the unit sphere is partitioned by an SOD is called a \emph{tile}.
\end{definition}

\begin{corollary}\label{p:tileanti}
In an SOD, no tile (including its boundary) contains two antipodal points.
\end{corollary}
\begin{proof}
If two antipodal points $p$ and $p'$ were in a same tile $T$ (or on its boundary), then by \cref{p:06} there would be a great semicircle with endpoints $p$ and $p'$ whose interior is entirely contained in $T$. However, this would contradict \cref{p:semi}.
\end{proof}

Since an SOD is a finite collection of arcs, tiles are \emph{spherical polygons}, whose boundaries have finitely many vertices and edges.

\begin{proposition}\label{p:tileedge}
In an SOD, any arc coincides with an edge of a tile.
\end{proposition}
\begin{proof}
Each arc $a$ is incident to tiles on both sides. There are multiple tiles on the same side of $a$ if and only if there are arcs hitting $a$ from that side. Thus, axiom~A3 implies that $a$ cannot have multiple tiles on both sides; in other words, a side of $a$ must have exactly one tile, and $a$ is an edge of that tile.
\end{proof}

The \emph{contact graph} of an SOD $\mathcal D$ is the undirected graph $(\mathcal D, \mathcal E)$, where $\mathcal E$ is the set of pairs of arcs $\{a,b\}\subseteq \mathcal D$ such that $a$ hits $b$. The following result implies that the contact graph of any SOD is 2-connected.

\begin{proposition}\label{p:07a}
Removing any one arc from an SOD and taking the union of the remaining arcs yields a connected subset of the unit sphere.
\end{proposition}
\begin{proof}
Let $\mathcal D$ be an SOD, let $a\in\mathcal D$, and let $D$ be the union of all arcs of $\mathcal D$ other than $a$. Our claim is equivalent to the statement that every connected component of $S\setminus D$ is simply connected, where $S$ is the unit sphere. 

Most of the connected components of $S\setminus D$ are interiors of tiles of $\mathcal D$, which are spherically convex (\cref{p:06}) and therefore simply connected. The only exceptions are the tiles incident to the interior of $a$, which constitute a single connected component $C$ of $S\setminus D$. Precisely, $C$ is the union of $a$ and the interiors of all tiles of $\mathcal D$ that have an edge in $a$.

We will prove that $C$ is contractible, and therefore simply connected. In fact, we will show that $C$ deformation-retracts onto $a$, which is obviously contractible. Consider a tile $T$ of $\mathcal D$ such that $e\subseteq a$ is an edge of $T$. Since both $T$ and $a$ are spherically convex, their intersection is also spherically convex, hence a geodesic arc. It follows that $T\cap a= e$, and therefore $\mathring T\cup a$ deformation-retracts onto $a$, where $\mathring T$ is the interior of $T$. This induces a deformation retraction of $C$ onto $C\setminus \mathring T$. By repeating the same process for all tiles with an edge in $a$, we conclude that $C$ deformation-retracts onto $a$.
\end{proof}

\begin{corollary}\label{p:07b}
The union of all the arcs in an SOD is a connected set.
\end{corollary}
\begin{proof}
Immediate from \cref{p:07a}.
\end{proof}

It is easy to see that \cref{p:07a} cannot be improved, as there are SODs that are disconnected by the removal of just two arcs. Indeed, starting from any SOD, it is always possible to add a new arc $a$, connecting two arcs $b$ and $c$, without violating the axioms of \cref{d:1}. Now, removing $b$ and $c$ from this new SOD isolates $a$ from the rest of the arcs. There are also irreducible SODs with this property, such as the one in \cref{fig:2conn}.

\begin{figure}[ht]
  \centering
  \includegraphics[scale=\figscale]{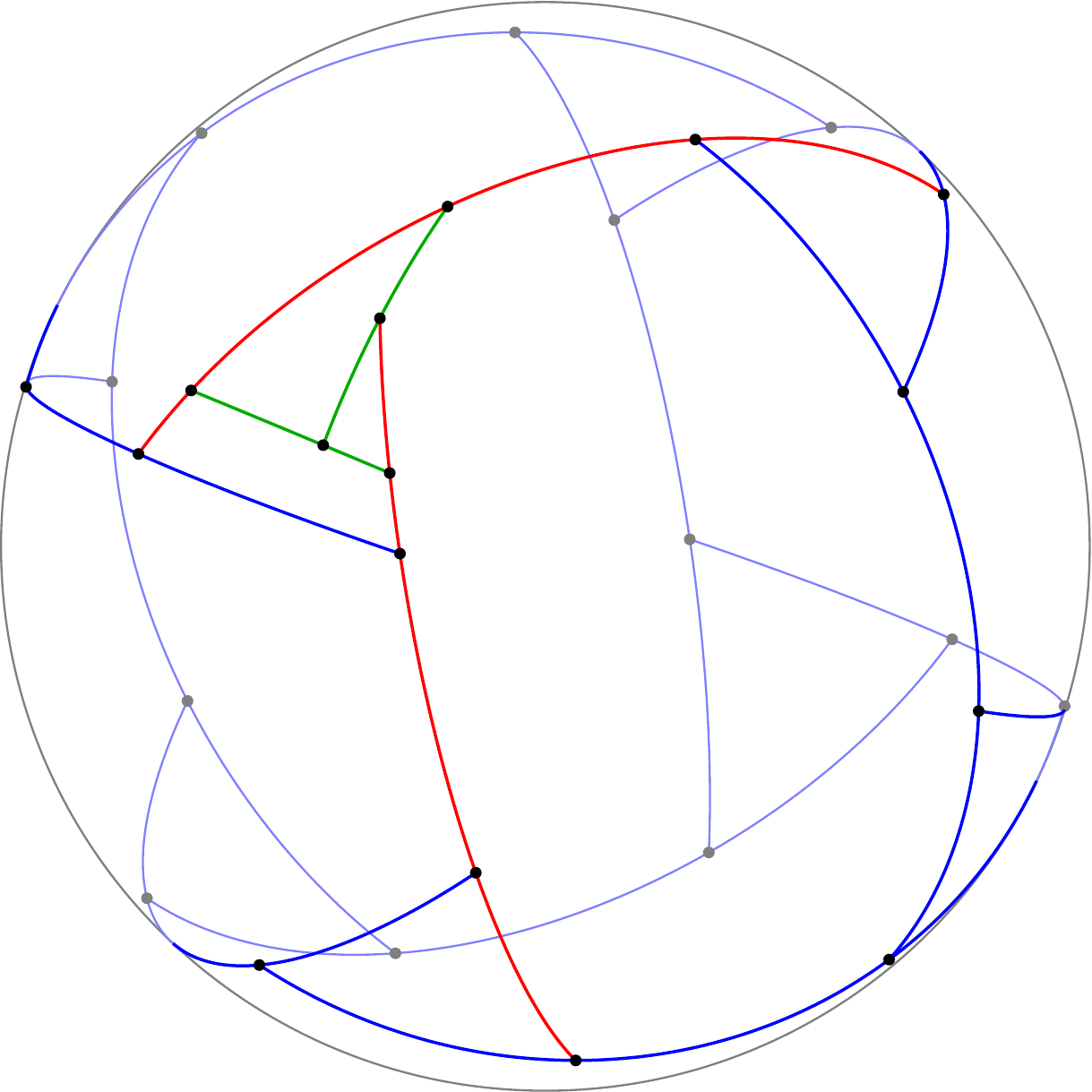}
  \caption{An irreducible SOD whose contact graph is not 3-connected. Removing the two arcs in red disconnects the arrangement into a green and a blue part.}
  \label{fig:2conn}
\end{figure}

There is a very clean relationship between the number of arcs in an SOD and the number of tiles: There are two more tiles than there are arcs.

\begin{proposition}\label{p:08}
An SOD with $n$ arcs partitions the unit sphere into $n+2$ tiles.
\end{proposition}
\begin{proof}
Every endpoint of an arc of an SOD divides the arc it hits into two sub-arcs. The set of these sub-arcs induces a spherical drawing of a planar graph with $k$ vertices and $n+k$ edges (if two arcs share an endpoint, this counts as a single vertex). Each face of this drawing coincides with a tile of the SOD. By Euler's formula, the number of faces is $(n+k)-k+2=n+2$.
\end{proof}

\section{Swirls\label{sec:4}}
There is a curious similarity between SODs and continuous vector fields on a sphere. According to the hairy ball theorem, ``it is impossible to comb a hairy ball without creating cowlicks''. Similarly, it is impossible to construct an SOD without creating ``swirls'', as we shall see in this section.

\begin{definition}
A \emph{swirl} in an SOD is a cycle of arcs, each of which feeds into the next (and such that the last feeds into the first), going either all clockwise or all counterclockwise. The \emph{degree} of a swirl is the number of arcs constituting it.
\end{definition}

Since no two arcs in an SOD can hit each other (\cref{p:05}), the minimum degree for a swirl is~3. \cref{fig:example} shows an SOD with six clockwise swirls and six counterclockwise swirls, all of degree~3.

We will now prove some basic properties of swirls. Note that for the second time in this paper, after \cref{p:tileedge}, we will be using axiom~A3.

\begin{proposition}\label{p:eyeconv}
In an SOD $\mathcal D$, let $\mathcal S$ be a swirl of degree $k$.
\begin{itemize}
\item The union of the $k$ arcs of $\mathcal S$ separates the unit sphere in two regions, exactly one of which is spherically convex; this region is a spherical $k$-gon called the \emph{eye} $E$ of $\mathcal S$.
\item The only points of intersection between pairs of arcs of $\mathcal S$ are the vertices of $E$.
\item The tiles of $\mathcal D$ adjacent to $E$ are exactly $k$; any two such tiles are either disjoint or intersect only along a single arc of $\mathcal S$.
\end{itemize}
\end{proposition}
\begin{proof}
Let $\mathcal S=(a_0, a_1, \dots, a_{k-1})$, and let $x_i$ be the unique endpoint of $a_{i-1}$ lying in the interior of $a_i$, with $0\leq i< k$ (indices are taken modulo $k$), as shown in \cref{fig:swirlint}. By definition of swirl, each arc of $\mathcal S$ feeds into the next forming convex angles; moreover, turning in the same direction (either clockwise or counterclockwise) at every $x_i$ yields a cycle of arcs. We conclude that following the $a_i$'s in order traces out the boundary of a spherically convex $k$-gon with vertices $x_0$, $x_1$, \dots, $x_{k-1}$. Let $E$ be such a convex $k$-gon (drawn in yellow in \cref{fig:swirlint}).

We already know that any two consecutive arcs of $\mathcal S$ intersect at a vertex of $E$ and only there (\cref{p:05}). We will now prove that non-consecutive arcs $a_i$ and $a_j$ of $\mathcal S$ are disjoint. Since $E$ is spherically convex, it completely lies in one of the two hemispheres bounded by $a_i$, and in one of the two hemispheres bounded by $a_j$. Thus, $E$ lies in a spherical lune $L$ determined by $a$ and $b$, as shown in \cref{fig:simple}. Clearly, $a_i$ and $a_j$ can intersect only at the vertices $p$ and $p'$ of $L$, which are antipodal points. Note that, because of the way the arcs of a swirl are oriented, $a_i$ may be incident to a vertex of $L$, say $p$, but not to the other vertex $p'$. Likewise, $a_j$ may be incident to $p'$ but not to $p$, and thus $a_i$ and $a_j$ cannot intersect.

Since the arcs of $\mathcal S$ have no intersections away from of the boundary of $E$, the region of sphere external to $E$ is not disconnected by $\mathcal S$, and is therefore a unique non-convex connected component.

It remains to prove that the tiles of $\mathcal D$ adjacent to $E$ are exactly $k$ and may only intersect each other along arcs of $\mathcal S$. Due to axiom~A3, no arc of $\mathcal D$ lying outside of $E$ can hit the boundary of $E$ away from its vertices. Thus, there can be at most $k$ tiles adjacent to $E$, and indeed these are the $k$ tiles that have an $a_i$ as an edge, for some $0\leq i<k$ (recall that every arc coincides with an edge of a tile, due to \cref{p:tileedge}).

Let us define $x_i'$ as the point antipodal to $x_i$, with $0\leq i< k$, and extend each $a_i$ in the direction opposite to $x_{i+1}$ until it reaches $x_i'$, as shown in \cref{fig:swirlint}. Our proof that non-adjacent swirl arcs are disjoint also implies that these arc extensions partition the sphere into exactly $k+2$ regions: $E$, the polygon $x_0'x_1'\dots x_{k-1}'$ (which is congruent to $E$ and antipodal to it), and $k$ spherical lunes, each of which entirely contains an arc $a_i$ and has $x_{i+1}$ and $x_{i+1}'$ as vertices.

Note that the tile of $\mathcal D$ adjacent to $E$ that has $a_i$ as an edge must be contained in one of these $k$ spherical lunes, namely the one whose vertices are $x_{i+1}$ and $x_{i+1}'$. Thus, the pairs of tiles corresponding to non-consecutive $a_i$'s must be disjoint, because their respective lunes are disjoint. On the other hand, any two tiles corresponding to adjacent lunes intersect along an arc $a_i$. Since tiles are spherically convex, their intersection must be a sub-arc of $a_i$.
\end{proof}

\begin{figure}[ht]
  \centering
  \includegraphics[scale=\figscale]{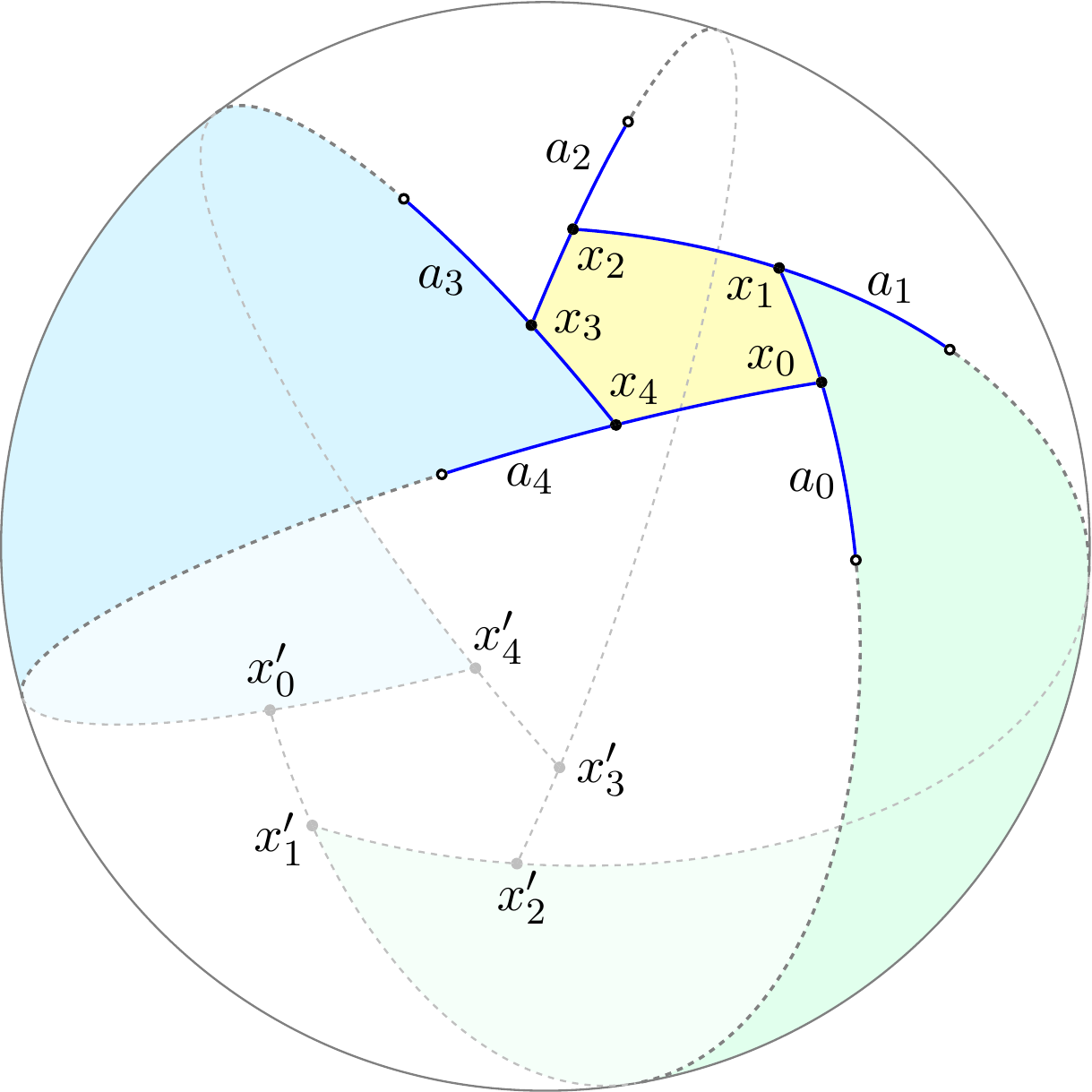}
  \caption{A counterclockwise swirl with its eye in yellow. The arcs constituting the swirl cannot intersect outside of the eye's vertices, and the two colored spherical lunes are disjoint.}
  \label{fig:swirlint}
\end{figure}

\begin{figure}[ht]
  \centering
  \includegraphics[scale=\figscale]{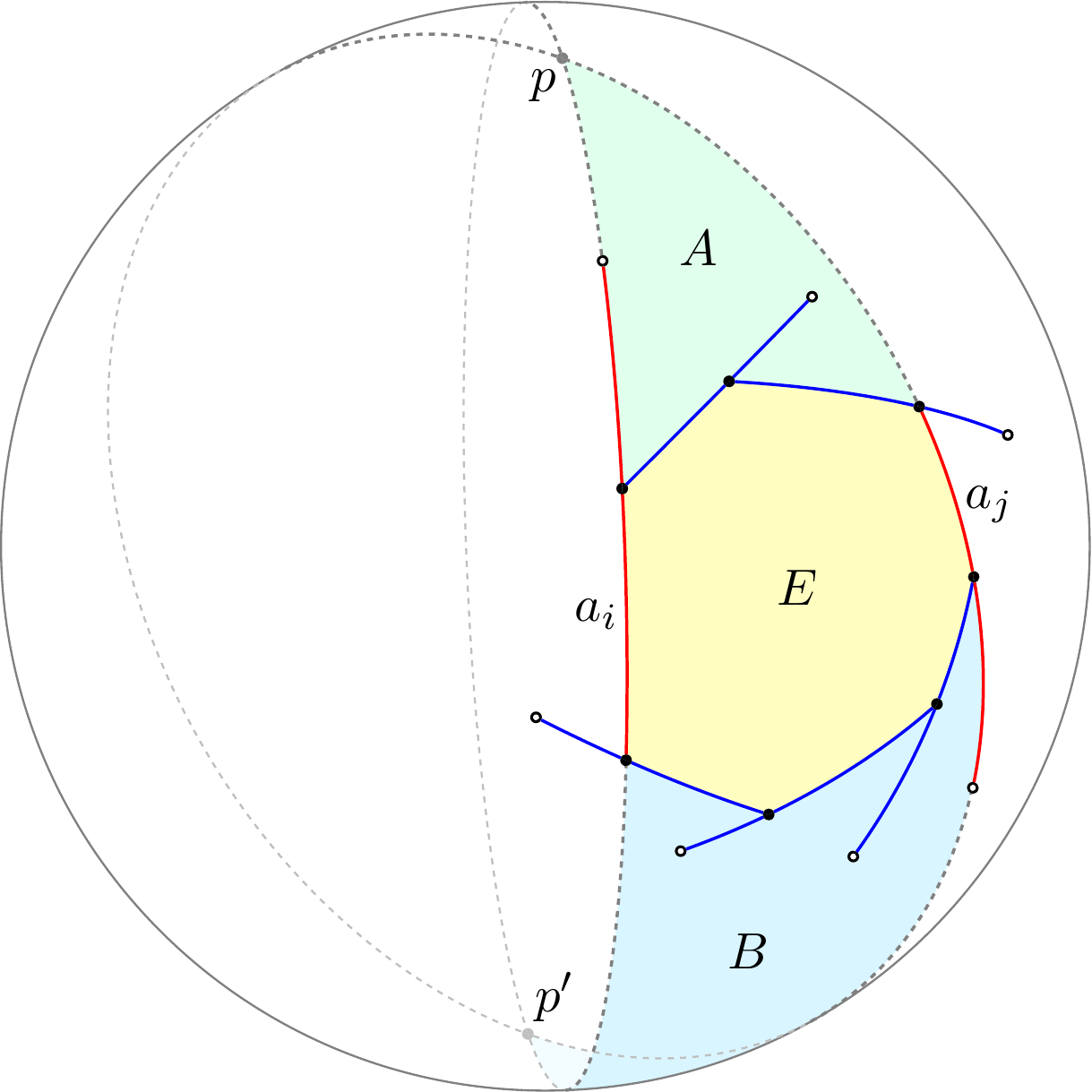}
  \caption{Any two non-consecutive arcs of a swirl are disjoint.}
  \label{fig:simple}
\end{figure}

Observe that, in an irreducible SOD, the eye of each swirl coincides with a single tile; in general, the eye of a swirl is a union of tiles, as there may be inner arcs.

\begin{proposition}\label{p:eyebound}
In an SOD, if the eyes of two distinct swirls have intersecting interiors, then their boundaries are disjoint.
\end{proposition}
\begin{proof}
Let $\mathcal S=(a_0, a_1, \dots, a_{k-1})$ and $\mathcal S'=(b_0, b_1, \dots, b_{k'-1})$ be two distinct swirls with eyes $E$ and $E'$, respectively. Note that, although the two swirls may share some arcs, not all arcs of $\mathcal S$ can be arcs of $\mathcal S'$, and vice versa (this is a simple consequence of \cref{p:05} and the definition of swirl).

Assume for a contradiction that $E$ and $E'$ have intersecting interiors and intersecting boundaries. Then, as we trace out the perimeter of $E'$ following the arcs of $\mathcal S'$ in order, we eventually reach a $b_i$ lying entirely in $E$ that hits an $a_j$, as shown in \cref{fig:swirlout}. Thus, $b_{i+1}=a_j$. This immediately implies that $\mathcal S$ and $\mathcal S'$ cannot be concordant, e.g., both clockwise. For otherwise, continuing to follow the arcs of $\mathcal S'$ clockwise we would reach $b_{i+2}=a_{j+1}$, $b_{i+3}=a_{j+2}$, etc., contradicting the fact that $\mathcal S$ and $\mathcal S'$ are distinct.

Hence, we may assume without loss of generality that $\mathcal S$ is a clockwise swirl and $\mathcal S'$ is a counterclockwise swirl. Now, continuing to trace out the perimeter of $E'$ following the arcs of $\mathcal S'$ counterclockwise from $b_{i+1}=a_j$ immediately leads outside of $E$ (see the green arrow starting at $b_i$ in \cref{fig:swirlout}).

We claim that once our walk around the perimeter of $E'$ has exited $E$, it can no longer reach its interior, contradicting the fact that $b_i$ is internal to $E$. Indeed, as soon as a counterclockwise walk reaches the boundary of $E$ at a certain $a_{i'}$, whether coming from the side of $a_{i'-1}$ or from the opposite side,\footnote{Note that the latter case cannot occur in an SOD, because it violates axiom~A3. We chose to present a more general proof which does not use axiom~A3, at the cost of slightly complicating the analysis.} it is immediately ``pushed back'' outside $E$, as illustrated by the green arrows in \cref{fig:swirlout}.
\end{proof}

\begin{figure}[ht]
  \centering
  \includegraphics[scale=\figscale]{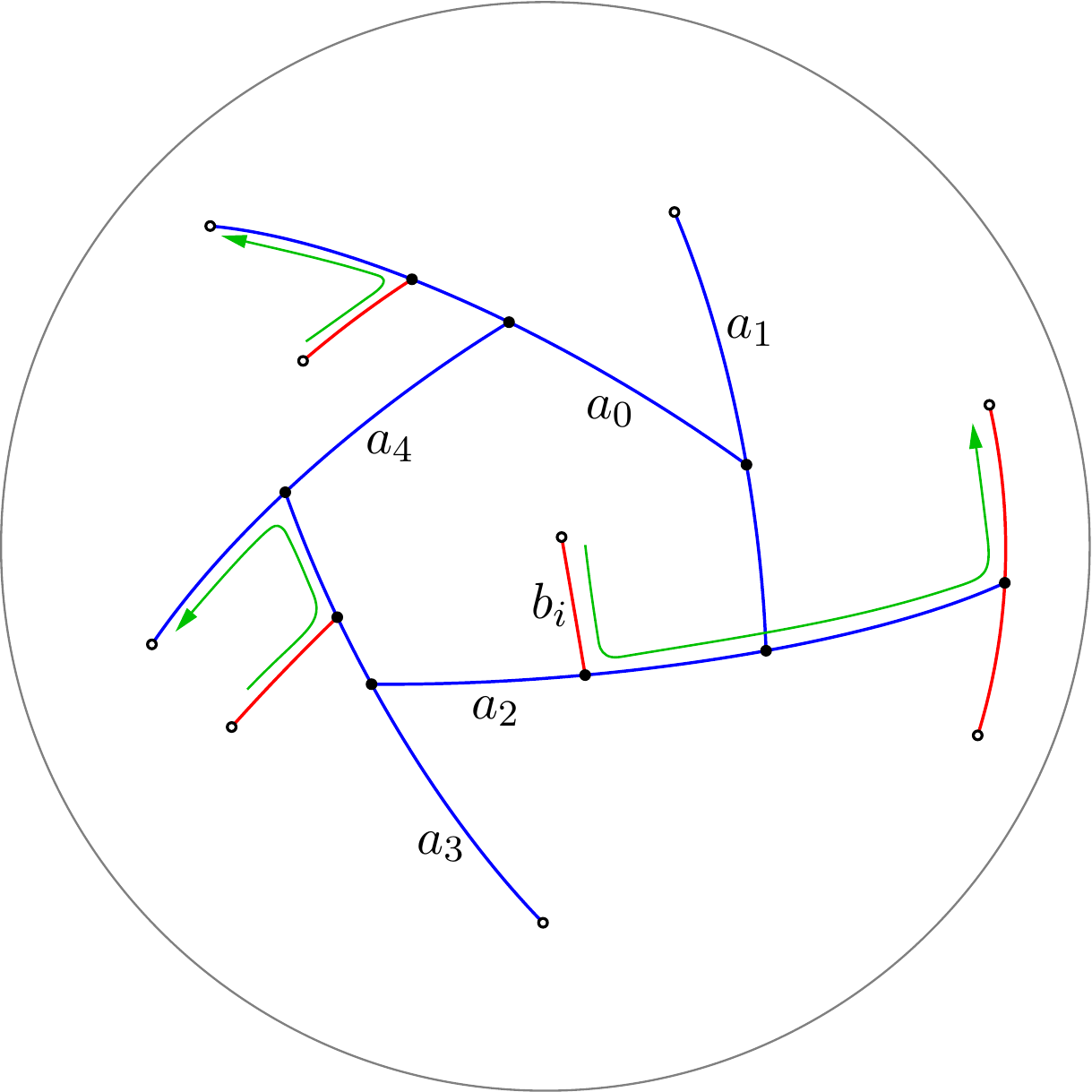}
  \caption{No two distinct swirls can have intersecting interiors and intersecting boundaries.}
  \label{fig:swirlout}
\end{figure}

\cref{p:eyebound} implies that any two eyes in an SOD are either internally disjoint or one is strictly contained in the other. Observe that, if we exclude ``degenerate'' SODs in which some arcs may share an endpoint, the statement of \cref{p:eyebound} is strengthened to ``distinct eyes have disjoint boundaries''. The inclusion of degenerate SODs, however, implies that some eyes may be internally disjoint but have a common vertex.

We will now study the combinatorics of swirls within an SOD by analyzing a structure called ``swirl graph''.

\begin{definition}
The \emph{swirl graph} of an SOD $\mathcal D$ is the undirected multigraph on the set of swirls of $\mathcal D$ having an edge between two swirls for every arc in $\mathcal D$ shared by the two swirls.
\end{definition}

\begin{figure}[ht]
  \centering
  \includegraphics[scale=\figscale]{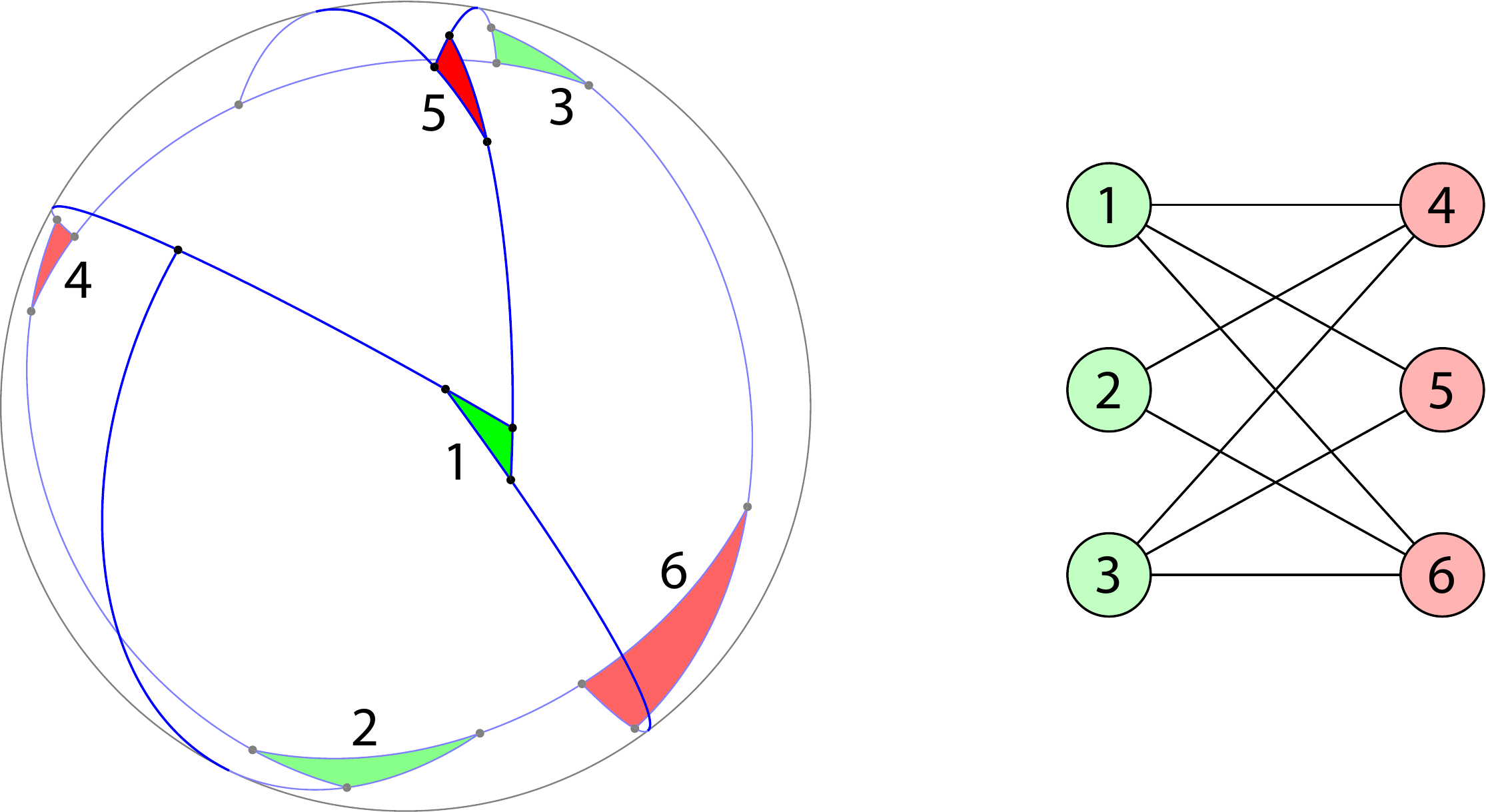}
  \caption{An SOD and its swirl graph.}
  \label{fig:sg}
\end{figure}

In \cref{fig:sg}, the eyes of clockwise swirls are colored green, and the eyes of counterclockwise swirls are colored red. Observe that the swirl graph of this SOD is simple and bipartite; this is true in general, as the next theorem shows.

\begin{theorem}\label{t:graph}
The swirl graph of any SOD is a simple planar bipartite graph with non-empty partite sets.
\end{theorem}
\begin{proof}
The swirl graph is spherical, hence planar. It is bipartite, where the partite sets correspond to clockwise and counterclockwise swirls, respectively. Indeed, if the same arc is shared by two concordant swirls (say, clockwise), then it is hit by arcs from both sides, violating axiom~A3.

\begin{figure}[ht]
  \centering
  \includegraphics[scale=\figscale]{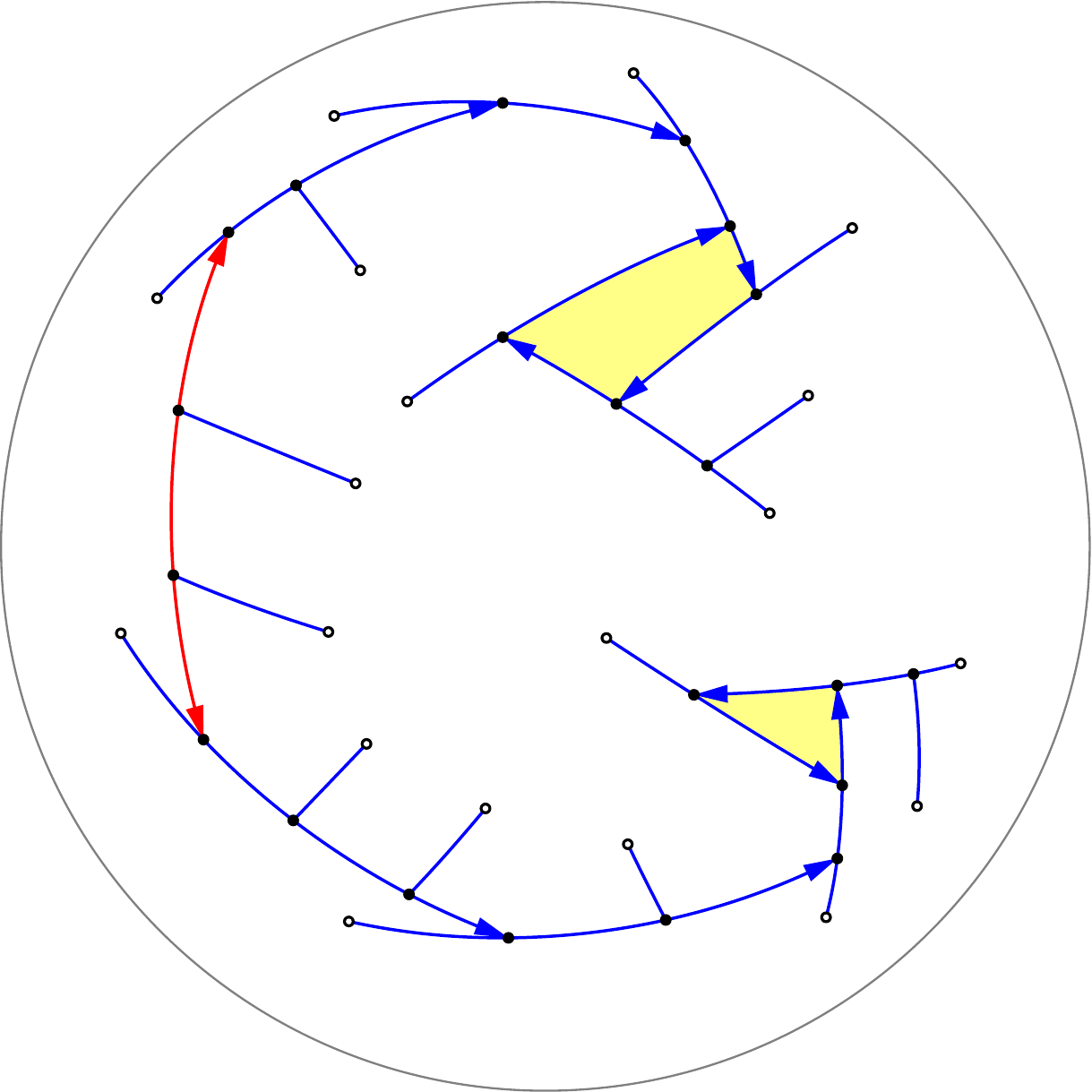}
  \caption{Finding swirls in an SOD.}
  \label{fig:s02}
\end{figure}

\cref{fig:s02} shows how to find a clockwise and a counterclockwise swirl in any SOD. For a clockwise swirl, start from any arc and follow it in any direction until it hits another arc. Then turn clockwise and follow this arc until it hits another arc, and so on. The sequence of arcs encountered is eventually periodic, and the period identifies a clockwise swirl. A counterclockwise swirl is found in a similar way.

To prove that the swirl graph is simple, assume for a contradiction that a swirl $\mathcal S$ shares two arcs $a_i$ and $a_j$ with another swirl $\mathcal S'$. Let $E$ and $E'$ be the eyes of $\mathcal S$ and $\mathcal S'$, respectively; observe that $a_i$ and $a_j$ are incident to both $E$ and $E'$. Let $\alpha$ and $\beta$ be the two great circles containing $a_i$ and $a_j$, respectively. As proved above, $\mathcal S$ and $\mathcal S'$ must be discordant, i.e., one is a clockwise swirl and the other is a counterclockwise swirl. Thus, $E$ and $E'$ lie on the same side of $a_i$. Furthermore, since both eyes are spherically convex (\cref{p:eyeconv}), they lie in the same hemisphere bounded by $\alpha$. For the same reason, $E$ and $E'$ lie in the same hemisphere bounded by $\beta$.

Therefore, there is a spherical lune $L$, bounded by $\alpha$ and $\beta$, such that both $E$ and $E'$ lie completely in $L$. Also, $E$ separates $L$ into two disjoint regions $A$ and $B$, such that some internal points of $a_i$ lie on the boundary of $A$ (but not on $B$) and some internal points of $a_j$ lie on the boundary of $B$ (but not on $A$), as shown in \cref{fig:simple}. Since the boundary of $E'$ connects $a_i$ and $a_j$, it must intersect the boundary of $E$. Hence, by \cref{p:eyebound}, $E$ and $E'$ must have disjoint interiors. Thus, as $E'$ is bounded by $a_i$, it must lie entirely in $A$. But $E'$ is also bounded by $a_j$, implying that it must lie in $B$, which is a contradiction. (In the special case where $a_i$ and $a_j$ are incident arcs, either $A$ or $B$ degenerates to the empty set, which of course cannot contain $E'$.)
\end{proof}

More is actually known about swirl graphs.

\begin{theorem}\label{t:4swirls}
Every SOD has at least four swirls.\qed
\end{theorem}

This result was announced in~\cite{tjc}, and a proof can be found in~\cite{arxiv}. From \cref{t:4swirls}, it easily follows that every SOD has at least eight arcs. On the other hand, \cref{fig:8arc} shows an example of an SOD with exactly eight arcs and exactly four swirls, which is therefore minimal.

\begin{figure}[ht]
  \centering
  \includegraphics[scale=\figscale]{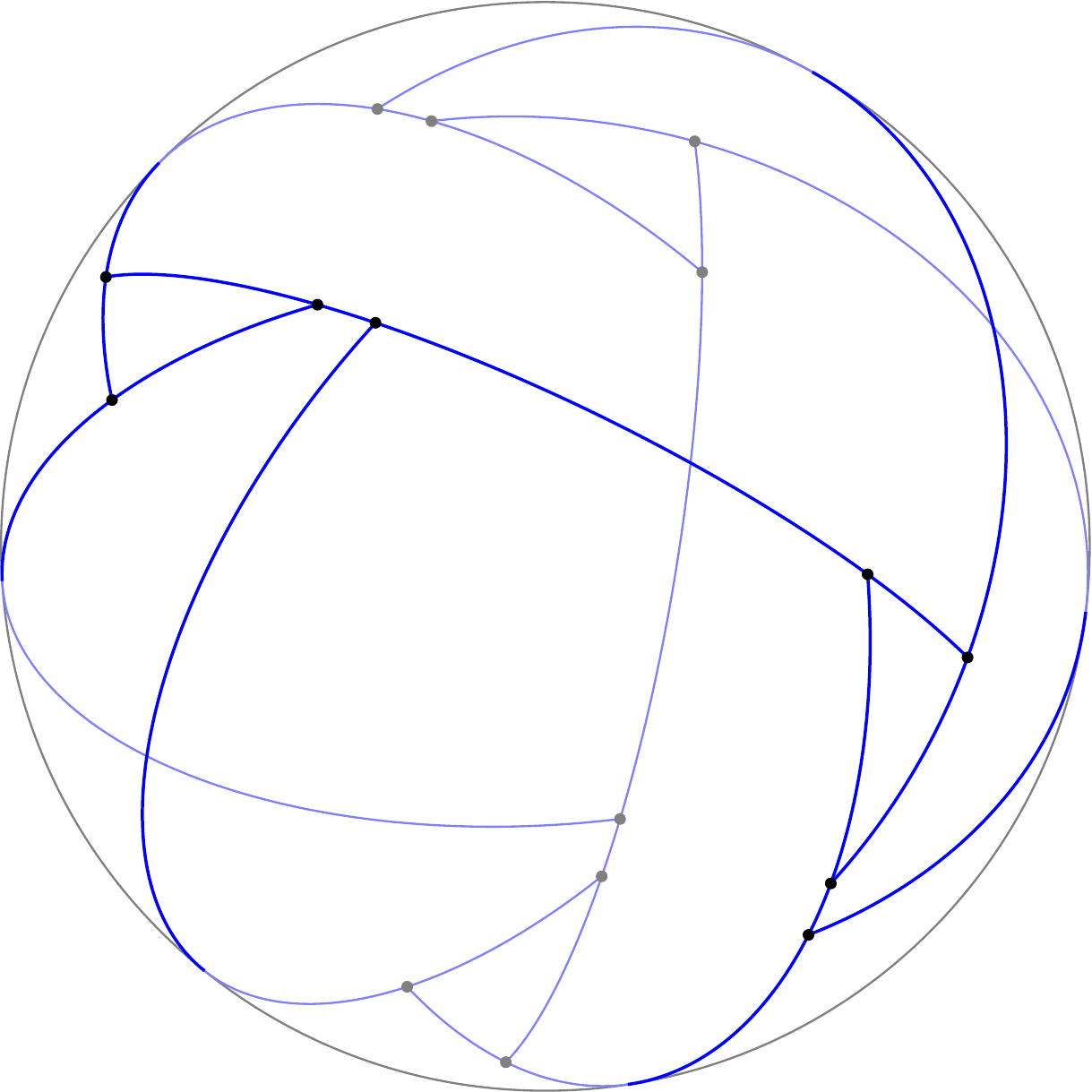}
  \caption{An SOD with eight arcs and four swirls.}
  \label{fig:8arc}
\end{figure}

It is not yet clear whether there are SODs with only one clockwise swirl (or only one counterclockwise swirl), but there is overwhelming evidence that this is not the case.

\begin{conjecture}
Every SOD has at least two clockwise and two counterclockwise swirls.
\end{conjecture}

\section{Swirling SODs\label{sec:5}}
This section is devoted to a special type of SODs whose arcs always meet forming swirls. Patterns arising in these SODs are found in modular origami, globe knots, rattan balls, etc.

\begin{definition}
An SOD is \emph{swirling} if each of its arcs is part of two swirls.
\end{definition}

We define the \emph{degree} of a (swirling) SOD as the the maximum degree of its swirls. Obviously, the degree of an SOD is at least~3. An example of a swirling SOD of degree~3 is found in \cref{fig:example}; further examples of swirling SODs are in \cref{fig:swirling}. All of these SODs were obtained from convex polyhedra or, more precisely, from convex subdivisions of the sphere, by a general process that we call \emph{swirlification}. As we will see in this section, there is a deep relationship between convex polyhedra and swirling SODs.

A \emph{convex subdivision} of the unit sphere is a partition into spherically convex polygons called \emph{faces} by a finite set of internally disjoint geodesic arcs called \emph{edges}, such that no two incident edges are collinear. As usual, the graph induced by the edges is called the \emph{1-skeleton} of the convex subdivision. Note that the visibility map of a convex polyhedron with respect to an internal point is the 1-skeleton of a convex subdivision of the unit sphere.

\begin{definition}
A \emph{swirlable} subdivision of the unit sphere is a convex subdivision where each face has an even number of edges.
\end{definition}

\begin{proposition}\label{p:tiling}
A convex subdivision of the unit sphere is swirlable if and only if its 1-skeleton is bipartite.
\end{proposition}
\begin{proof}
The 1-skeleton is bipartite if and only if its has no odd cycles, which is true if and only if each face has an even number of edges.
\end{proof}

Note that we can also obtain a swirlable subdivision of the sphere by taking the dual of a subdivision whose vertices have even degree, or by truncating it. More generally, we have the following.

\begin{proposition}\label{p:dualtrunc}
A convex subdivision of the unit sphere is swirlable if and only if its truncated dual is swirlable.
\end{proposition}
\begin{proof}
Observe that the operations of truncation and taking the dual of a convex subdivision of the sphere preserve the convexity of the subdivision. To conclude the proof, we only have to show that a convex subdivision has even-sided faces if and only if its truncated dual does.

\cref{fig:pp4} illustrates the procedure of truncating the dual of a convex subdivision $\mathcal S_1$ of the sphere. The edges of $\mathcal S_1$ are drawn in blue, and the ones in red are the edges of its dual $\mathcal S_2$. Note that $\mathcal S_2$ has a vertex in each face of $\mathcal S_1$, and vice versa. Moreover, the degree of each vertex of $\mathcal S_2$ is equal to the number of edges of the face of $\mathcal S_1$ containing it. Truncating $\mathcal S_2$ amounts to drawing the green edges and deleting the dashed portions of the red edges, as well as all red vertices; let $\mathcal S_3$ be the new subdivision. The last operation has two effects on the set of faces. On one hand, it doubles the edges of each face of $\mathcal S_2$; thus, these faces are always even-sided. On the other hand, it adds a new face around each vertex of $\mathcal S_2$, with a number of edges equal to its degree. Hence, these faces are even-sided if and only if the faces of $\mathcal S_1$ are even-sided.
\end{proof}

\begin{figure}[ht]
  \centering
  \includegraphics[scale=\figscale]{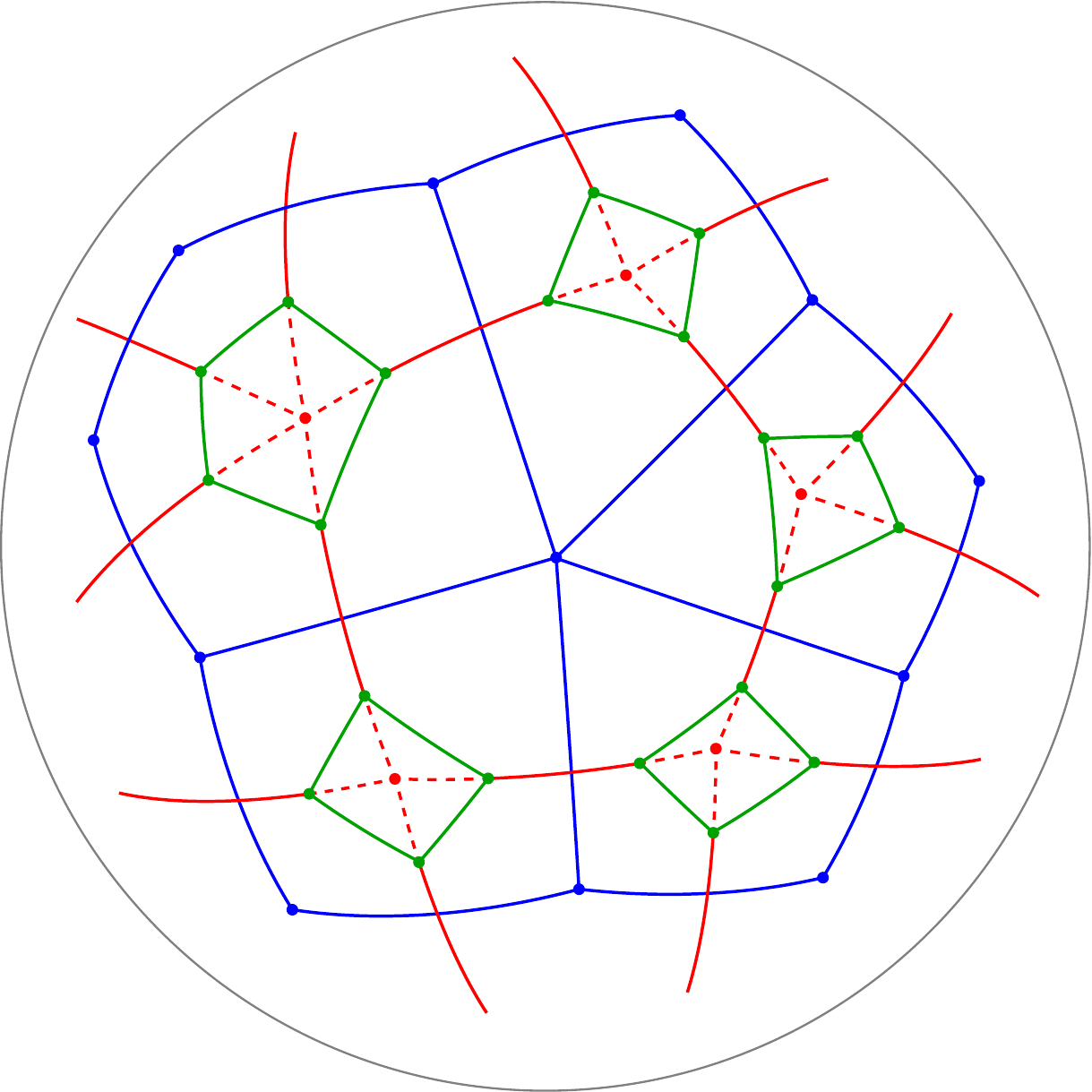}
  \caption{Constructing the truncated dual of a convex subdivision of the sphere.}
  \label{fig:pp4}
\end{figure}

Given a swirlable subdivision of the sphere, the operation of turning each of its vertices into a swirl, going clockwise or counterclockwise according to the bipartition of the 1-skeleton, is called \emph{swirlification}. As we will prove next, the swirlification operation allows us to turn any swirlable subdivision of the sphere into a swirling SOD. In fact, we will prove something more: In the case of swirling SODs of degree~3, the inverse operation is also possible, although not always trivial.

This will lead us to a characterization of the swirl graphs of swirling SODs as the 3-regular bipartite polyhedral graphs (a graph is \emph{polyhedral} if can be realized as the 1-skeleton of a convex polyhedron). Note that, by a well-known theorem of Steinitz, these graphs can also be described as the 3-regular 3-vertex-connected planar bipartite graphs.\footnote{Steinitz's theorem states that the polyhedral graphs are precisely the 3-vertex-connected planar graphs (see~\cite[Chapter~13]{convex}).} These are also known as \emph{Barnette graphs}~\cite{barnette}, because Barnette conjectured in 1969 that they are Hamiltonian.

We will start by illustrating the swirlification operation.

\begin{proposition}\label{t:swirlification}
The 1-skeleton of any swirlable subdivision of the unit sphere is isomorphic to the swirl graph of a swirling SOD.
\end{proposition}
\begin{proof}
Let $\mathcal S$ be any swirlable subdivision of the unit sphere. Since each edge of $\mathcal S$ is a geodesic arc, we can perturb its endpoints within a small-enough $\epsilon$-neighborhood without making them antipodal. This operation allows us to create a small swirl in lieu of each vertex of $\mathcal S$, while moving each edge by at most $\epsilon$ (this is possible in particular because the faces of $\mathcal S$ are convex, and hence do not have reflex angles). If the swirls go clockwise or counterclockwise according to the bipartition of the 1-skeleton of $\mathcal S$, and if $\epsilon$ is small enough, the resulting set of geodesic arcs is a swirling SOD. Since every arc in a swirling SOD is an edge of its swirl graph, the swirl graph must be $G$.
\end{proof}

\cref{fig:swirling} shows two applications of \cref{p:dualtrunc,t:swirlification}. The two top SODs are swirlifications of a 6-trapezohedron and a truncated 6-antiprism; note that $k$-trapezohedron and $k$-antiprism are dual polyhedra. The two bottom SODs are swirlifications of a rhombic triacontahedron and a truncated icosidodecahedron; the rhombic triacontahedron and the icosidodecahedron are dual polyhedra. This process can be continued indefinitely: By repeatedly taking the truncated dual of a swirlable subdivision of the sphere and swirlifying it, one obtains an infinite family of swirling SODs.

\begin{figure}[!ht]
  \centering
  \includegraphics[width=\linewidth]{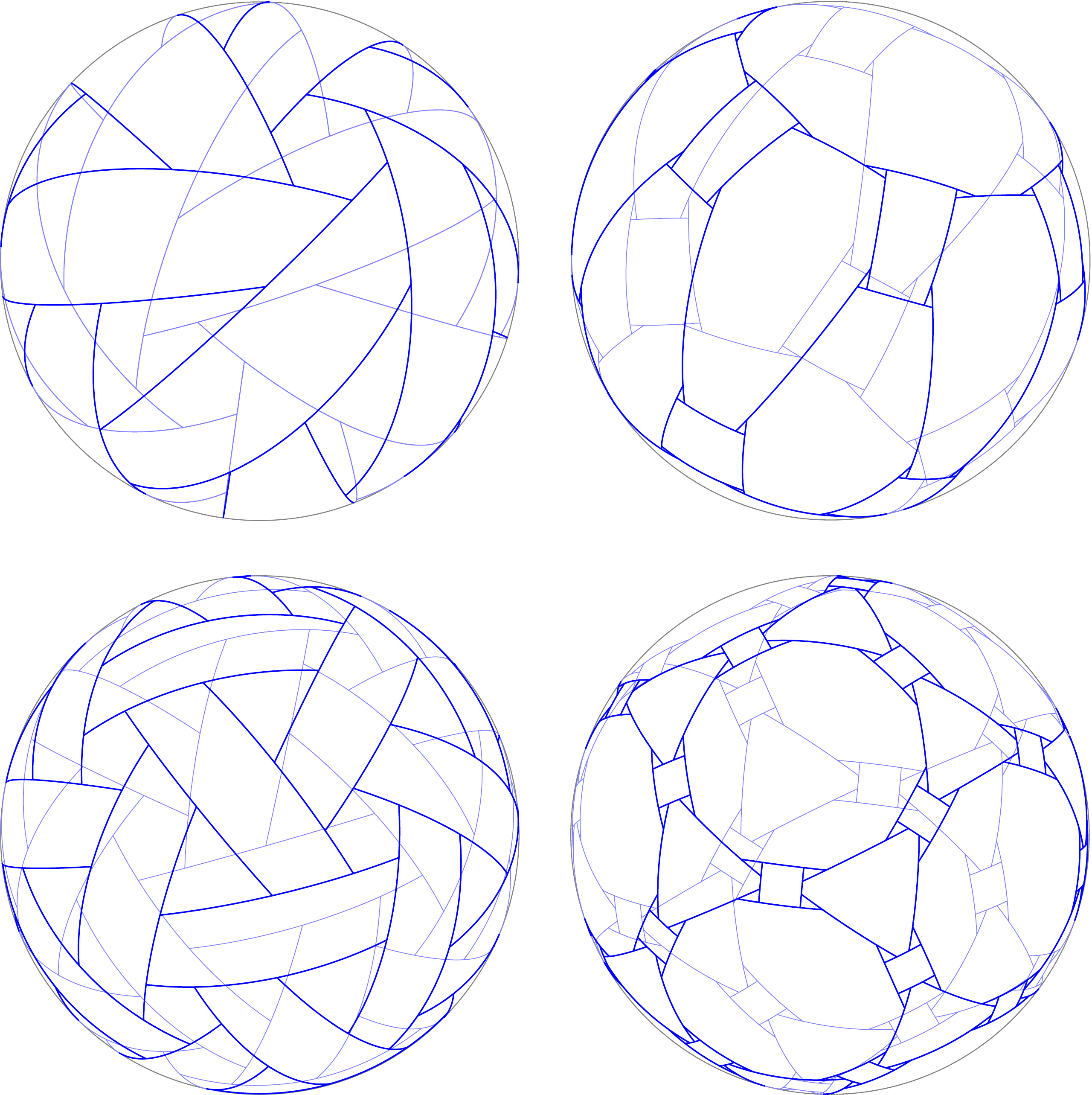}
  \caption{Examples of the swirlification method developed in \cref{sec:5} to produce swirling SODs from convex polyhedra with a bipartite 1-skeleton (or, equivalently, from swirlable subdivisions of the unit sphere). The pictures show swirling SODs resulting from a 6-trapezohedron, a truncated 6-antiprism, a rhombic triacontahedron, and a truncated icosidodecahedron, respectively.}
  \label{fig:swirling}
\end{figure}

We will now discuss the connectivity of the swirl graphs of swirling SODs. Observe that \cref{p:07a} immediately implies that the swirl graph of any swirling SOD is 2-edge-connected. The two following theorems strengthen this result.

\begin{theorem}\label{t:swirling1}
The swirl graph of any swirling SOD is 2-vertex-connected.
\end{theorem}
\begin{proof}
Let $\mathcal S=(a_0,a_1,\dots,a_{k-1})$ be any swirl in a swirling SOD $\mathcal D$, and let $x_i$ be the point where $a_{i-1}$ feeds into $a_i$, with $0\leq i<k$ (indices are taken modulo $k$). We define an \emph{eyelash} $\ell_i$ of $\mathcal S$ as the maximal open geodesic arc $x_iy_i\subseteq a_i\setminus x_{i+1}x_i$ that contains no endpoints of arcs of $\mathcal D$, as show in \cref{fig:2conn2}. By \cref{p:eyeconv}, no two arcs of $\mathcal S$ intersect away from $x_0$, $x_1$, \dots, $x_{k-1}$; therefore, $\mathcal S$ has exactly $k$ disjoint eyelashes $\ell_0$, $\ell_1$, \dots, $\ell_{k-1}$.

Let $D$ be the union of the arcs in $\mathcal D$. Our claim is equivalent to the statement that removing the perimeter of $E$ and the $k$ eyelashes of $\mathcal S$ from $D$ does not disconnect $D$. In turn, since no arcs of $\mathcal D$ lie in the interior of $E$,\footnote{It is intuitively true that no eye of a swirling SOD has inner arcs, and therefore $E$ coincides with a single tile of $\mathcal D$. We will not formally prove it here because in \cref{sec:6} we will give a self-contained proof of a more general statement: All \emph{uniform} SODs (hence all swirling SODs) are irreducible.} this is equivalent to the statement that every component of $R=(S\setminus D) \cup E \cup \ell_0\cup \ell_1\cup \dots \cup \ell_{k-1}$ is simply connected, where $S$ is the unit sphere.

Most of the connected components of $R$ are interiors of tiles of $\mathcal D$, which are spherically convex, hence simply connected. The only exception is the component $C$ (the yellow region in \cref{fig:2conn2}) consisting of the union of the eye $E$, the $k$ eyelashes of $\mathcal S$, and the interiors of the tiles $T_0$, $T_1$, \dots, $T_{k-1}$ of $\mathcal D$ adjacent to $E$. By \cref{p:eyeconv}, any two tiles adjacent to $E$ are either disjoint or intersect only along a single arc of $\mathcal S$. Thus, the tile $T_i$ is only adjacent to $T_{i-1}$ along (the closure of) $\ell_i$ and to $T_{i+1}$ along (the closure of) $\ell_{i+1}$. Also, $T_i$ is adjacent to $E$ along $x_ix_{i+1}$. Therefore, the common boundary between the open disk $\mathring T_i$ and $C\setminus \mathring T_i$ is $\ell_i\cup \ell_{i+1}\cup x_ix_{i+1}$, which is simply connected.

It follows that $C$ deformation-retracts onto $C\setminus \mathring T_i$ for every $0\leq i<k$. By composing these $k$ deformation retractions, we obtain a deformation retraction of $C$ onto $E \cup \ell_0\cup \ell_1\cup \dots \cup \ell_{k-1}$, which is obviously contractible. We conclude that $C$ is simply connected, as desired.
\end{proof}

\begin{figure}[ht]
  \centering
  \includegraphics[scale=\figscale]{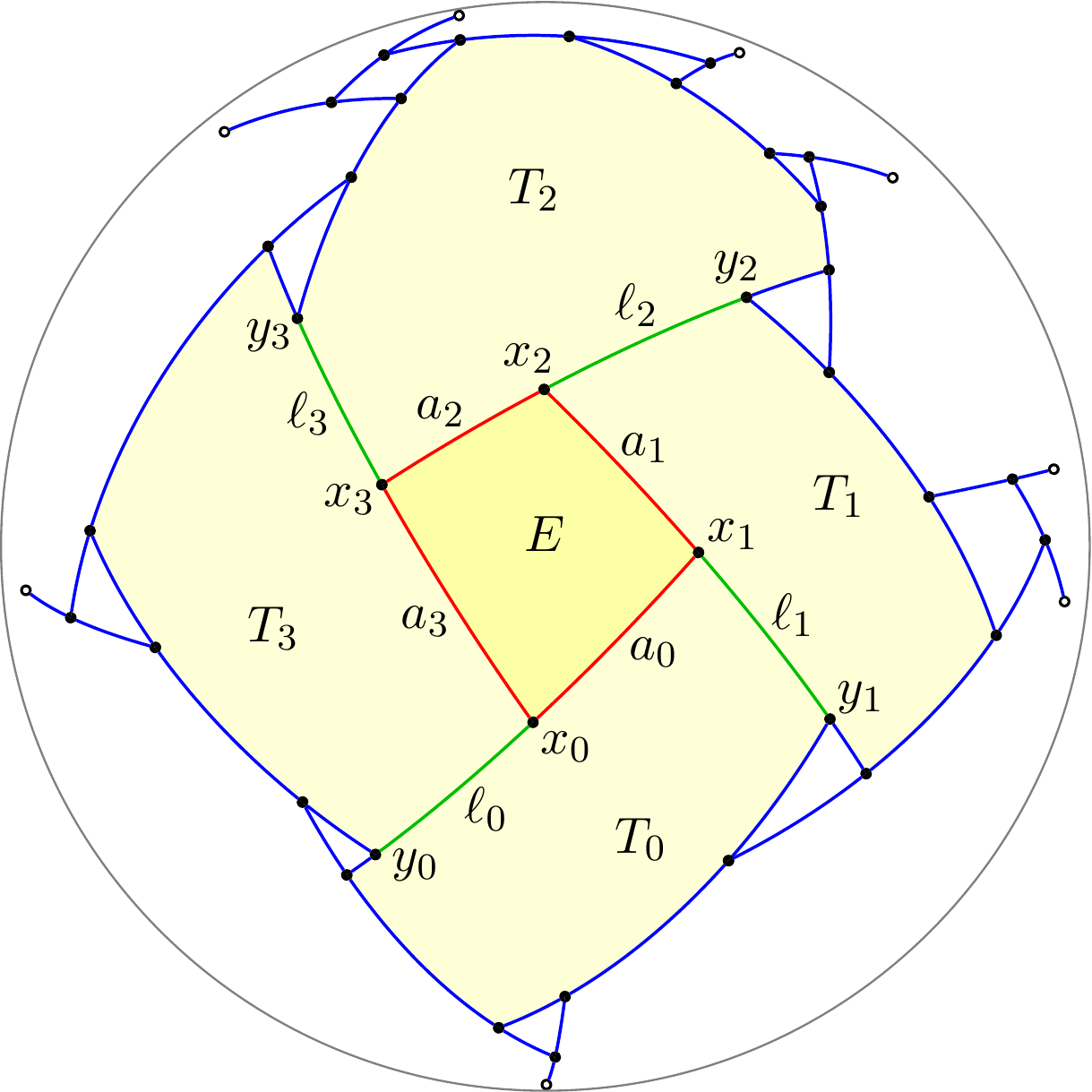}
  \caption{A swirl in a swirling SOD with its eye $E$, eyelashes $\ell_i$, and neighboring tiles $T_i$.}
  \label{fig:2conn2}
\end{figure}

\begin{theorem}\label{t:swirling2}
The swirl graph of any swirling SOD of degree~3 is 3-vertex-connected.
\end{theorem}
\begin{proof}
Let $\mathcal S_1$ and $\mathcal S_2$ be two swirls of a swirling SOD $\mathcal D$ of degree~3, and let $E_1$ and $E_2$ be their respective eyes. We define the ``eyelashes'' of the two swirls as in the proof of \cref{t:swirling1}, and we denote as $P_1$ the union of $E_1$ and the eyelashes of $\mathcal S_1$, while $P_2$ is the union of $E_2$ and the eyelashes of $\mathcal S_2$. Our claim is that every component of $R=(S\setminus D)\cup P_1\cup P_2$ is simply connected, where $S$ is the unit sphere and $D$ is the union of the arcs of $\mathcal D$ (again, this relies on the fact that $E_1$ and $E_2$ have no inner arcs, as the SOD is swirling).

Most components of $R$ are the interiors of tiles of $\mathcal D$, which are spherically convex, hence simply connected. \cref{fig:cases} sketches the three possible configurations for the remaining components of $R$; the boundaries of $P_1$ and $P_2$ are represented in red. In the first case, $\mathcal S_1$ and $\mathcal S_2$ do not share any arcs and no tile of $\mathcal D$ is adjacent to both $E_1$ and $E_2$. So, there is a component of $R$ containing $E_1$ and a distinct component containing $E_2$. Reasoning as in \cref{t:swirling1}, we conclude that both components are simply connected.

In the second case, $\mathcal S_1$ and $\mathcal S_2$ share an arc $a$ with endpoints $p_1\in E_1$ and $p_2\in E_2$, as in \cref{fig:cases} (center). By \cref{t:graph}, $\mathcal S_1$ and $\mathcal S_2$ cannot share more than one arc. Thus, there are two tiles $A$ and $B$ of $\mathcal D$ that are incident to the interior of $a$ (on opposite sides of it) and adjacent to both $E_1$ and $E_2$. Note that $A$ and $B$ are distinct, due to \cref{p:eyeconv}. Since both swirls have degree~3, there is exactly one more tile $C$ adjacent to $E_1$ and one more tile $D$ adjacent to $E_2$. Observe that the (unique) component of $R$ containing (the interiors of) all of these tiles is simply connected, provided that $C$ and $D$ are distinct tiles. So, assume for a contradiction that $C=D$, and thus both points $p_1$ and $p_2$ lie on the boundary of $C$. Since $p_1$ and $p_2$ are endpoints of the geodesic arc $a$, they are not antipodal, and therefore $a$ is the unique geodesic arc connecting them. On the other hand, $C$ is spherically convex, and hence the geodesic arc $p_1p_2=a$ must lie within $C$, contradicting the fact that the interior of $a$ does not border $C$ (cf.~\cref{p:eyeconv}).

The only remaining case is shown in \cref{fig:cases} (right): $\mathcal S_1$ and $\mathcal S_2$ share no arcs, but there is a tile $A$ adjacent to both $E_1$ and $E_2$. Let $B_1$ and $B_2$ be the two remaining tiles (other than $A$) adjacent to $E_1$, and let $C_1$ and $C_2$ be the two remaining tiles adjacent to $E_2$. There is a unique component of $R$ containing (the interiors of) all of these tiles, and it is simply connected, provided that all the tiles involved are distinct. In particular, we only have to rule out the case that $B_i=C_j$ for $i,j\in\{1,2\}$. Since both swirls have degree~3, there is an arc $a$ of $\mathcal S_1$ incident to both $A$ and $B_i$, and there is an arc $a'$ of $\mathcal S_2$ incident to both $A$ and $C_j$. Note that $a$ and $a'$ are distinct arcs, because $\mathcal S_1$ and $\mathcal S_2$ share no arcs. Assume for a contradiction that $B_i=C_j$. Then, the two tiles $A$ and $B_i$, which are both adjacent to $E_1$, are also incident along the two arcs $a$ and $a'$, which contradicts \cref{p:eyeconv}. We conclude that all components of $R$ are simply connected.
\end{proof}

\begin{figure}[ht]
  \centering
  \includegraphics[scale=1]{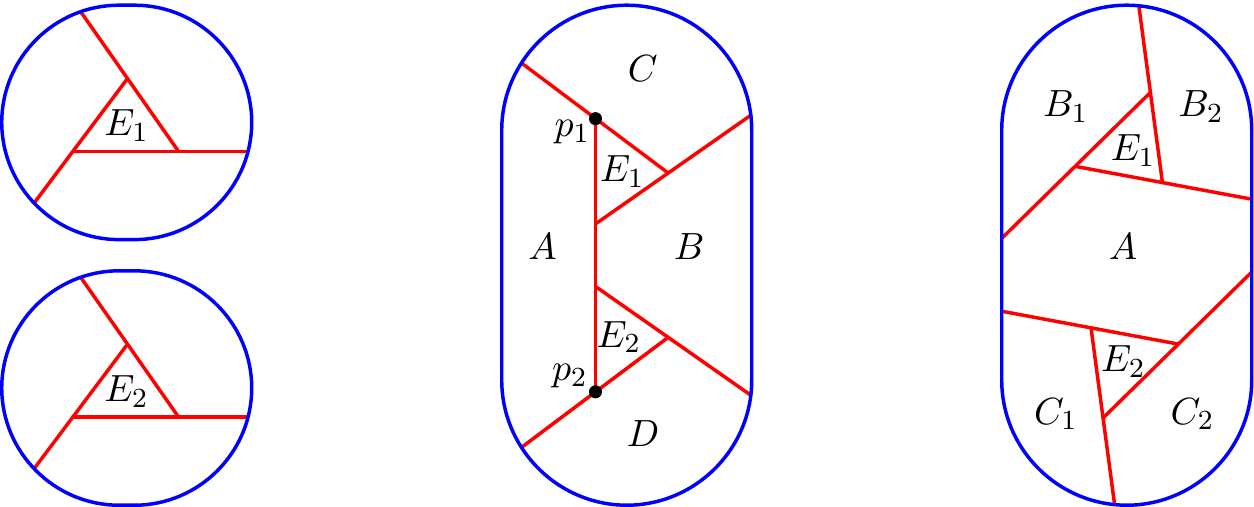}
  \caption{Sketch of the possible configurations of two swirls in a swirling SOD of degree~3.}
  \label{fig:cases}
\end{figure}

We remark that \cref{t:swirling1,t:swirling2} cannot be improved, as there are swirling SODs of degree~4 whose swirl graph is not 3-vertex-connected, such as the one in \cref{fig:counter2}. Similar counterexamples of any degree greater than~3 can be constructed, as well. 

\begin{figure}[ht]
  \centering
  \includegraphics[scale=\figscale]{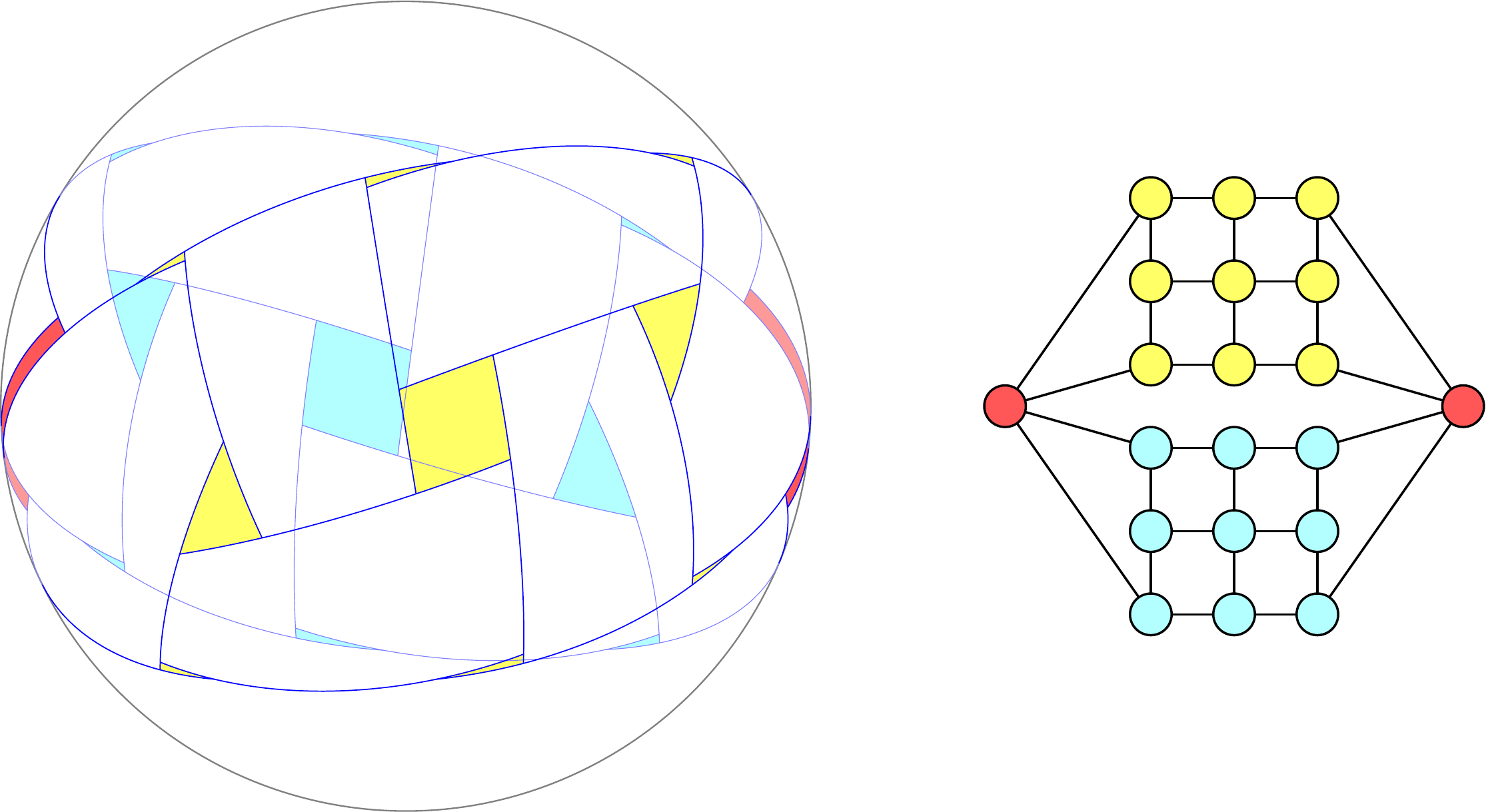}
  \caption{A swirling SOD of degree~4 (left) whose swirl graph is not 3-vertex-connected (right). Deleting the two red swirls disconnects the swirl graph into a yellow and a teal component, representing the swirls on the front side and on the back side of the sphere, respectively.}
  \label{fig:counter2}
\end{figure}

As a corollary to \cref{t:swirling2}, we can now characterize the swirl graphs of swirling SODs of degree~3.

\begin{corollary}\label{c:swirling}
The swirl graphs of swirling SODs of degree~3 are precisely the Barnette graphs.
\end{corollary}
\begin{proof}
Let $G$ be a Barnette graph, i.e., a 3-regular 3-connected planar bipartite graph. We will construct a swirling SOD of degree~3 whose swirl graph is $G$. By Steinitz's theorem, since $G$ is 3-connected and planar, there is a convex polyhedron $\mathcal P$ whose 1-skeleton is isomorphic to $G$. Radially projecting the 1-skeleton of $\mathcal P$ onto a unit sphere whose center is internal to $\mathcal P$ yields a convex subdivision $\mathcal S$ of the sphere whose 1-skeleton is isomorphic to $G$. Since $G$ is bipartite, \cref{p:tiling} implies that $\mathcal S$ is swirlable. Finally, due to \cref{t:swirlification}, the swirlification of $\mathcal S$ yields a swirling SOD whose swirl graph is $G$. Moreover, since $G$ is 3-regular, all swirls of this SOD must have degree~3.

Conversely, let $G$ be the swirl graph of a swirling SOD of degree~3. By \cref{t:graph}, $G$ is planar and bipartite; by \cref{t:swirling2}, $G$ is 3-connected. Also, since the SOD is swirling, each of its swirls has a number of neighbors in $G$ equal to its degree. As the SOD has degree~3, and no swirl in an SOD can have degree lower than~3, it follows that all swirls in the SOD have degree~3, and the swirl graph $G$ is 3-regular. We conclude that $G$ is a Barnette graph.
\end{proof}

An easy consequence of Steinitz's theorem and \cref{c:swirling} is that the swirl graphs of swirling SODs of degree~3 are precisely the 1-skeletons of swirlable subdivisions of the sphere whose vertices have degree~3. In other words, the converse of \cref{t:swirlification} is true, and the swirlification operation can be reversed, provided that all swirls have degree~3. However, the reversal process is not always trivial, and this explains why our proof is not direct but relies on Steinitz's theorem, which is a deep result in polyhedral combinatorics.

For instance, one might be tempted to reverse the swirlification operation by simply ``contracting'' the eye of each swirl of a swirling SOD of degree~3 into an interior point in order to obtain a swirlable subdivision of the sphere. Unfortunately, this is not always possible, and \cref{fig:counterswirl} shows a counterexample consisting of a swirling SOD of degree~3 with a 3-fold rotational symmetry around the center of a large triangular eye $E$. Suppose we replace the eye of each swirl with an interior point and move the arcs of the SOD accordingly to form a subdivision of the sphere. No matter how an interior point $p\in E$ is chosen, the subdivision always has a non-convex face (the perimeter of such a face is drawn in red in \cref{fig:counterswirl}), hence it is not a swirlable subdivision.

\begin{figure}[ht]
  \centering
  \includegraphics[scale=\figscale]{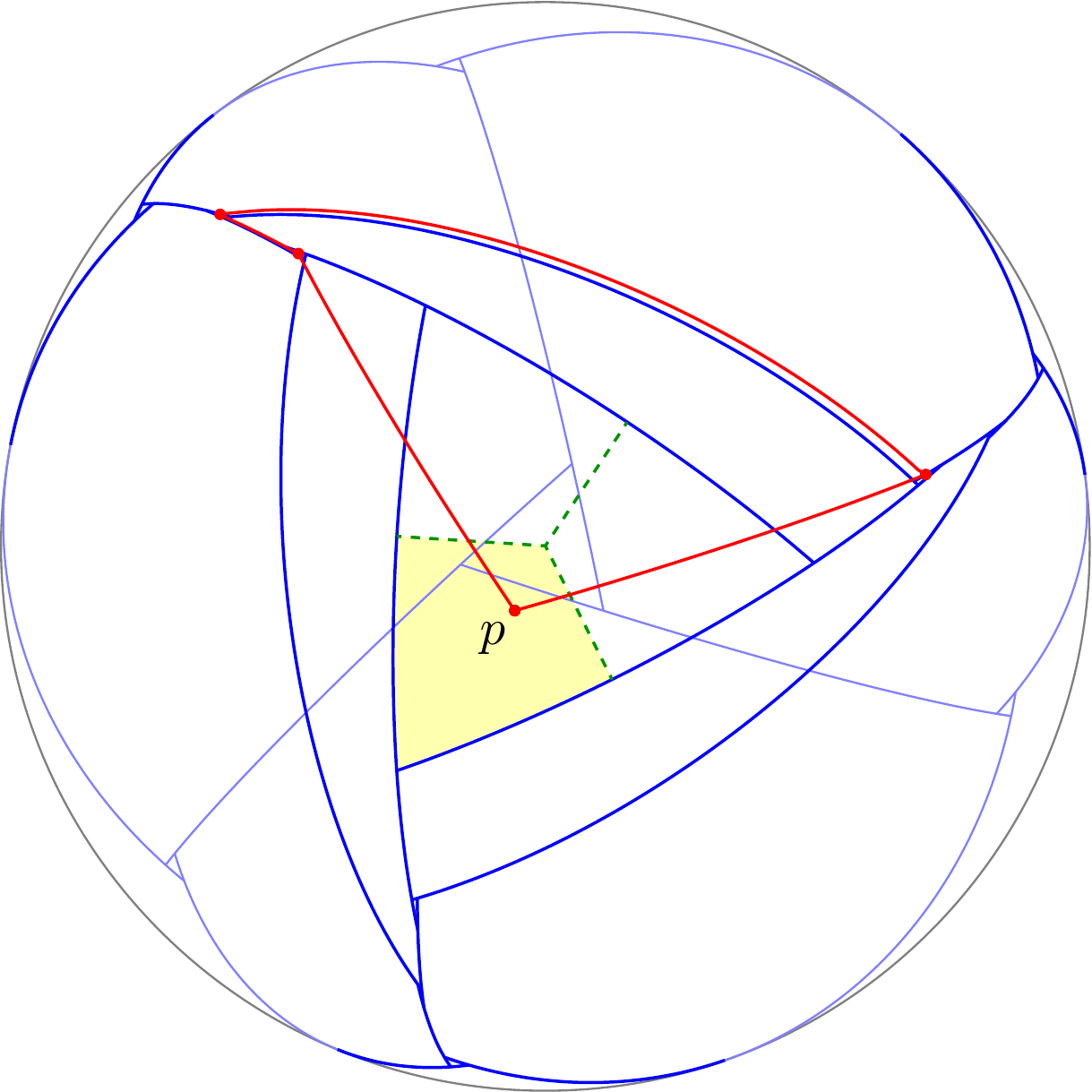}
  \caption{A swirling SOD of degree~3 where contracting the eye of each swirl into an interior point necessarily yields a non-convex subdivision of the sphere.}
  \label{fig:counterswirl}
\end{figure}

\section{Uniform SODs\label{sec:6}}
We now turn to a class of SODs that generalizes the swirling ones.

\begin{definition}
An SOD is \emph{uniform} if every arc blocks the same number of arcs.
\end{definition}

\begin{proposition}\label{p:uniform}
For an SOD $\mathcal D$, the following are equivalent.
\begin{enumerate}
\item $\mathcal D$ is uniform.
\item Every arc of $\mathcal D$ blocks exactly two arcs.
\item Every arc of $\mathcal D$ blocks at least two arcs.
\item Every arc of $\mathcal D$ blocks at most two arcs.
\end{enumerate}
\end{proposition}
\begin{proof}
By \cref{p:04}, each arc hits exactly two distinct arcs. Hence, each arc blocks two arcs on average. Thus, if all arcs blocks the same number of arcs, this number must be two. For the same reason, if every arc blocks at most two arcs (or at least two arcs), it must block exactly two arcs.
\end{proof}

\begin{proposition}\label{p:uniirred}
Every uniform SOD is irreducible.
\end{proposition}
\begin{proof}
Let $\mathcal D$ be a uniform SOD, and assume that there is a proper subset of arcs $\mathcal D'\subset \mathcal D$ that is itself an SOD. By \cref{p:07b}, $\mathcal D$ is connected; thus, removing arcs from $\mathcal D$ causes some arcs to block fewer than two arcs. Since $\mathcal D$ is uniform, it follows that the arcs of $\mathcal D'$ block fewer than two arcs on average, contradicting \cref{p:04}.
\end{proof}

\begin{corollary}\label{c:eyeuni}
In a uniform SOD, the eye of each swirl coincides with a single tile.
\end{corollary}
\begin{proof}
If the interior of the eye of a swirl contains some arcs, then such arcs can be removed without violating the SOD axioms. Hence, such an SOD is not irreducible, and by \cref{p:uniirred} it cannot be uniform.
\end{proof}

\begin{theorem}\label{t:uniform}
Every swirling SOD is uniform.
\end{theorem}
\begin{proof}
In a swirling SOD, each arc $a$ is part of two distinct swirls. By \cref{t:graph}, these two swirls share no arcs other than $a$, and hence $a$ must block one arc from each of them. Therefore, every arc in a swirling SOD blocks at least two arcs, and by \cref{p:uniform} the SOD is uniform.
\end{proof}

The converse of \cref{t:uniform} is not true in general, as \cref{fig:u05} shows.

\begin{figure}[ht]
  \centering
  \includegraphics[scale=\figscale]{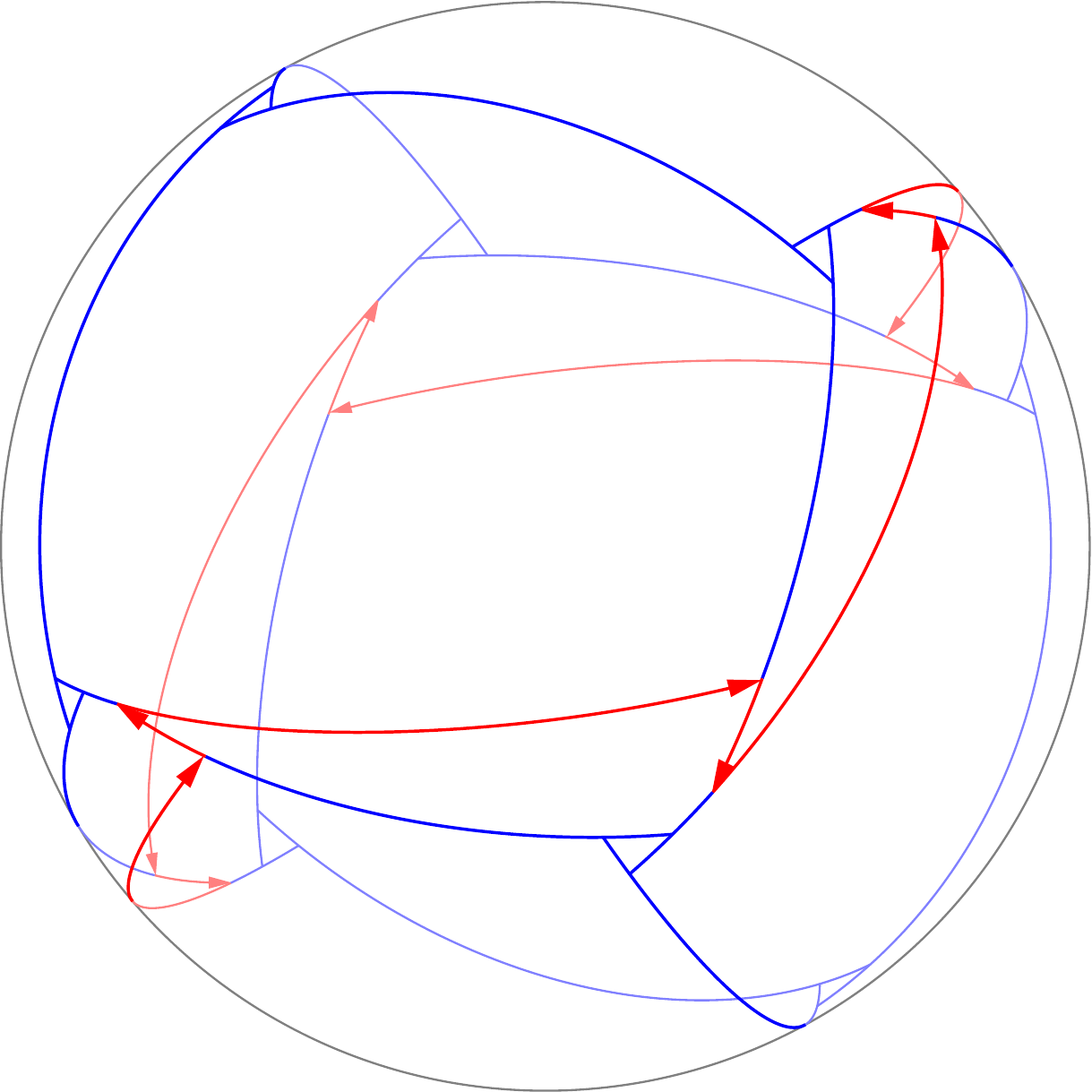}
  \caption{A uniform SOD that is not swirling.}
  \label{fig:u05}
\end{figure}

\begin{definition}
An endpoint of an arc of an SOD is called a \emph{swirling vertex} if it is incident to the eye of a swirl. A walk on an SOD is \emph{non-swirling} if it does not touch swirling vertices and, whenever it touches the interior of an arc, it follows it until it reaches one of its endpoints, without touching any other arc along the way. A cyclic non-swirling walk is called a \emph{non-swirling cycle}.
\end{definition}

Observe that there is a non-swirling cycle that covers all the non-swirling vertices of the SOD in \cref{fig:u05} (drawn in red). As the next theorem shows, this is not a coincidence.

\begin{theorem}\label{t:cycles}
In a uniform SOD where no two arcs share an endpoint, all non-swirling vertices are covered by disjoint non-swirling cycles.
\end{theorem}
\begin{proof}
Suppose that no two arcs of a uniform SOD share an endpoint. Consider a non-swirling walk $W$ terminating at a non-swirling vertex $x_i$, endpoint of an arc $a_i$, as \cref{fig:u09new} shows. We will prove that $W$ can be extended to a longer non-swirling walk in a unique way.

\begin{figure}[ht]
  \centering
  \includegraphics[scale=\figscale]{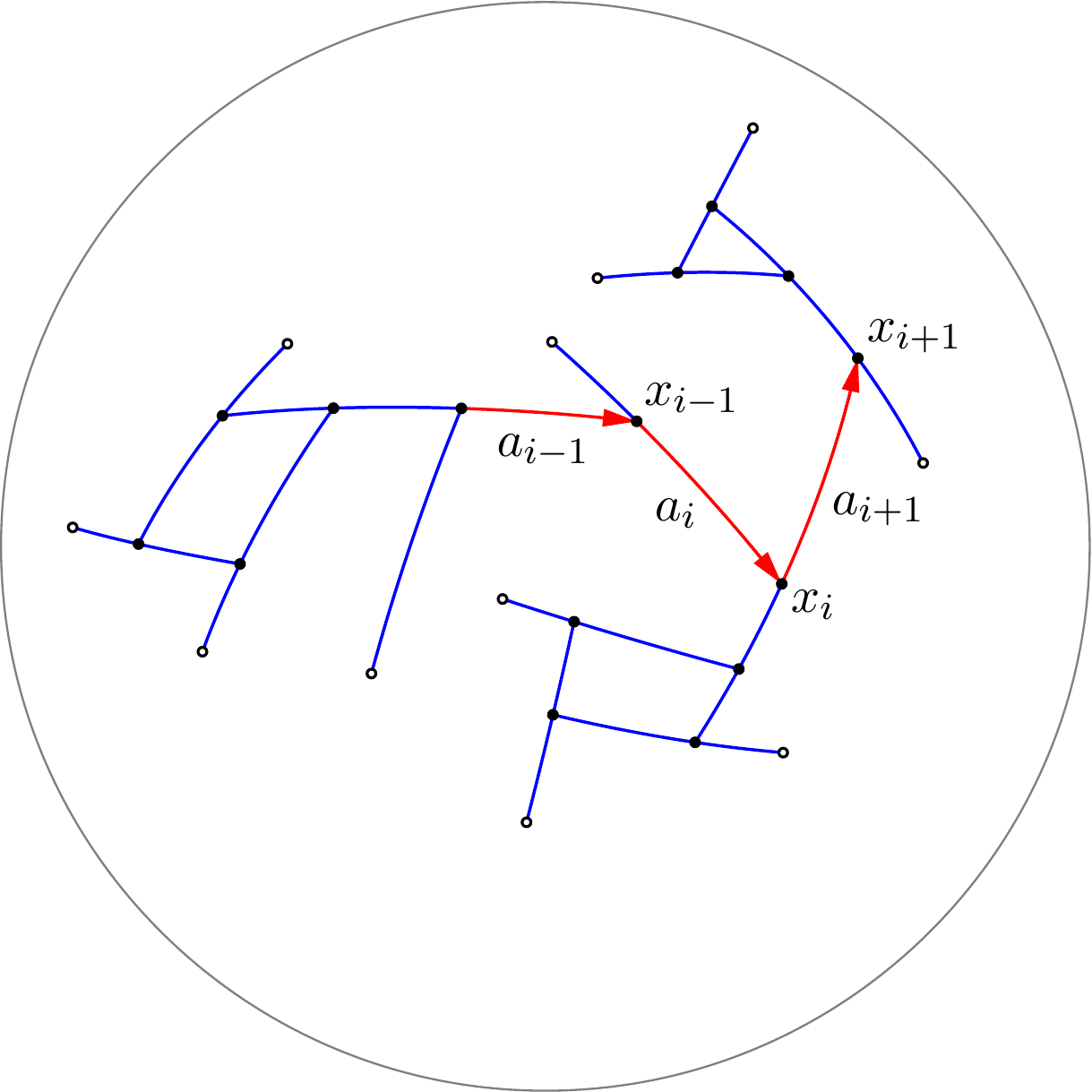}
  \caption{Proof of \cref{t:cycles}: A non-swirling walk can be extended indefinitely.}
  \label{fig:u09new}
\end{figure}

Let $a_{i+1}$ be the arc that blocks $a_i$ at $x_i$. Since exactly one arc other than $a_i$ hits $a_{i+1}$, and it does so at an endpoint other than $x_i$, there is exactly one endpoint of $a_{i+1}$, say $x_{i+1}$, that can be reached from $x_i$ without touching any arc other than $a_{i+1}$.

By definition of non-swirling walk, $x_{i+1}$ can be used to extend $W$ if and only if it is a non-swirling vertex. However, if $x_{i+1}$ were incident to a swirl's eye $E$, then an arc of that swirl would either hit $a_{i+1}$ between $x_{i}$ and $x_{i+1}$, contradicting the fact that $a_{i+1}$ blocks exactly two arcs, or it would hit $a_{i+1}$ on the other side of $x_{i}$, implying that $E$ contains the arc $a_i$ in its interior, which contradicts \cref{c:eyeuni}.

Hence, $W$ can be extended uniquely to a non-swirling walk. By a similar reasoning, we argue that $W$ can also be uniquely extended backwards to a non-swirling walk. Thus, $W$ is part of a unique non-swirling cycle. Now we conclude the proof by inductively repeating the same argument with any remaining non-swirling vertices.
\end{proof}

It is straightforward to extend \cref{t:cycles} to ``degenerate'' uniform SODs where arcs may share endpoints. However, in this case the portions of arcs that do not contribute to swirls are only partitioned into edge-disjoint cycles (as opposed to vertex-disjoint cycles).

We remark that we can construct uniform SODs with any number of unboundedly long non-swirling cycles. An example with two non-swirling cycles is shown in \cref{fig:u12}.

\begin{figure}[ht]
  \centering
  \includegraphics[scale=\figscale]{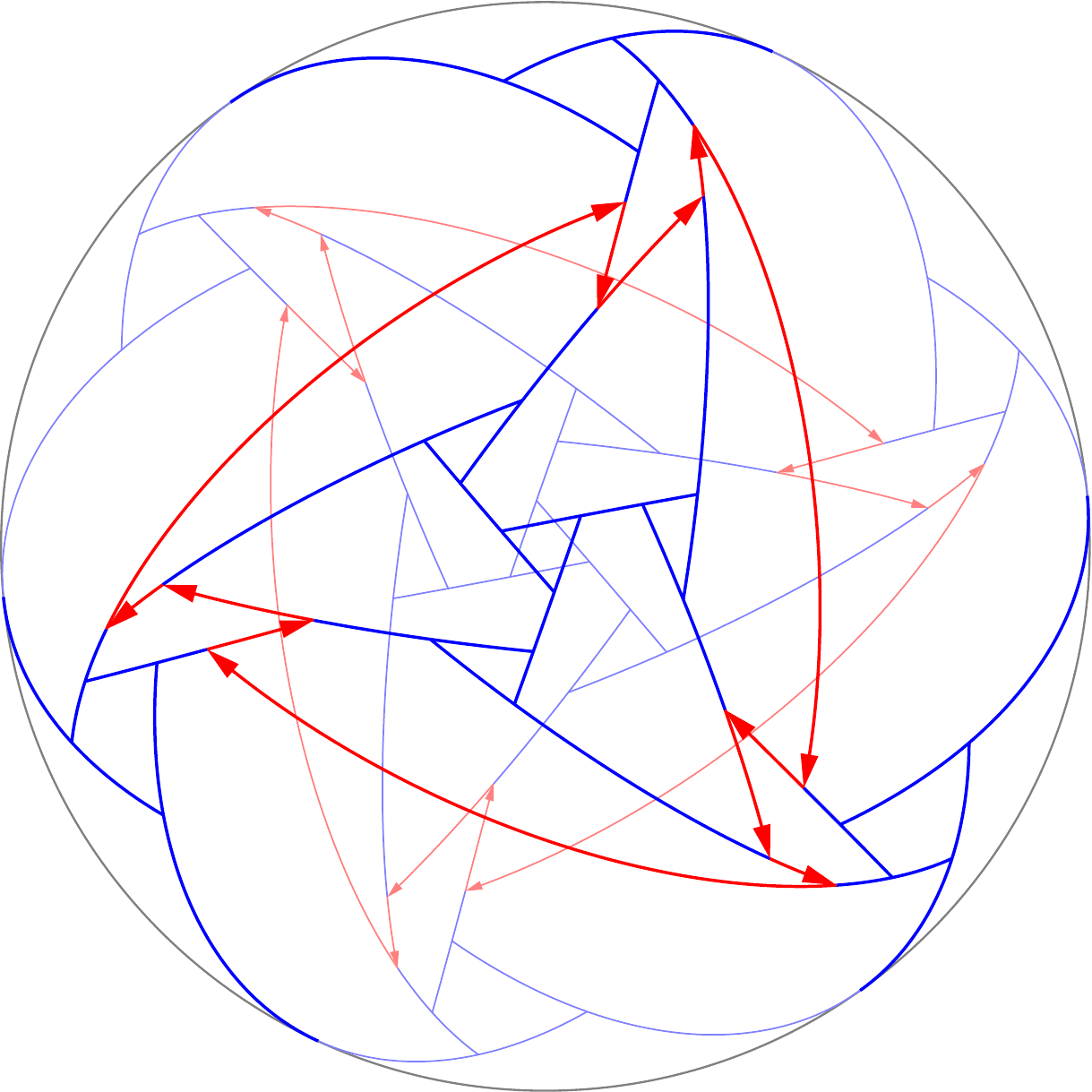}
  \caption{A uniform SOD with two non-swirling cycles.}
  \label{fig:u12}
\end{figure}

\section{Conclusions}\label{sec:7}
We introduced the theory of Spherical Occlusion Diagrams and studied their basic properties, while also discussing some applications to visibility-related problems in discrete and computational geometry.

Although we strongly believe \cref{c:1} to be true, i.e., not all irreducible SODs can be realized as polyhedra, a related and more subtle question can be asked, inspired by previous work on weaving patterns~\cite{basu,pach}. Namely, whether for every SOD $\mathcal D$ there is a \emph{combinatorially equivalent} SOD $\mathcal D'$ and a polyhedron $\mathcal P$ such that $\mathcal D'=S_\mathcal P$. In other words, does every class of combinatorially equivalent SODs contain a realizable instance?

We have introduced three noteworthy families of SODs: irreducible, uniform, and swirling. We proved that all swirling SODs are uniform, and all uniform SODs are irreducible. In the case of swirling SODs of degree~3 (i.e., where all swirls consist of exactly three arcs), we were able to characterize their swirl graphs as the Barnette graphs, that is, the 3-regular 3-connected planar bipartite graphs (\cref{c:swirling}). It remains an open problem to characterize the swirl graphs of other families of SODs. \cref{t:graph,t:swirling1,t:cycles} are steps in this direction.

It is somewhat surprising that there is a wealth of swirling SODs whose swirl graphs are not bipartite polyhedral graphs (cf.~\cref{fig:counter2}). Nonetheless, it is natural to wonder if the graphs in these two classes share any notable properties, such as the one established by \cref{t:swirling1} (i.e., being 2-vertex-connected).

Another common feature of bipartite polyhedral graphs and the swirl graphs of swirling SODs is that all such graphs have vertices of degree~3.\footnote{It is an elementary consequence of Euler's formula that any planar graph contains a triangle or a vertex of degree at most~3. Thus, a bipartite planar graph must have a vertex of degree at most~3, as it cannot contain triangles.} In turn, this immediately implies that any swirling SOD has a swirl of degree~3. This contrasts with the fact that there are (non-swirling) SODs containing only swirls of arbitrarily large degree.

\cref{t:cycles} reveals a structural connection between swirling and uniform SODs, which might be taken a step further. It seems to be possible to systematically transform any uniform SOD into a swirling SOD by ``sliding'' the endpoints of some arcs along other arcs and ``merging'' coincident arcs. Making this observation rigorous is left as a direction for future work.

More generally, we may wonder which SODs can be transformed into swirling ones by sequences of elementary operations on arcs (defining suitable ``elementary operations'' is in itself an open problem). The SOD in \cref{fig:8arc} shows that the question is not trivial. Indeed, this is the unique configuration of any SOD with eight or fewer arcs; since the SOD itself is not swirling, it cannot be transformed into a swirling one by means of operations that only rearrange or merge arcs.

T\'oth pointed out an interesting similarity between SODs and \emph{one-sided rectangulations}, which are subdivisions of an axis-aligned rectangle into axis-aligned rectangles such that every maximal inner edge coincides with a side of one of the rectangles.\footnote{This property is akin to our \cref{p:tileedge}, and the term ``one-sided'' is reminiscent of our axiom~A3.} It is known that any one-sided rectangulation is \emph{area-universal}, i.e., any assignment of areas to its rectangles can be realized by a combinatorially equivalent rectangular layout~\cite{eppstein,merino}. We wonder if a similar property holds for SODs, and to what extent the theory of one-sided rectangulations can provide insights for the study of SODs.

\bibliography{spherical}

\begin{thebibliography}{10}

\bibitem{basu}
Saugata Basu, Raghavan Dhandapani, and Richard Pollack.
\newblock On the realizable weaving patterns of polynomial curves in {$\mathbb
  R^3$}.
\newblock In {\em Proceedings of the 12th International Symposium on Graph
  Drawing (GD 2004)}, pages 36--42, 2004.

\bibitem{benbernou}
Nadia~M. Benbernou, Erik~D. Demaine, Martin~L. Demaine, Anastasia Kurdia,
  Joseph O'Rourke, Godfried~T. Toussaint, Jorge Urrutia, and Giovanni
  Viglietta.
\newblock Edge-guarding orthogonal polyhedra.
\newblock In {\em Proceedings of the 23rd Canadian Conference on Computational
  Geometry (CCCG 2011)}, pages 461--466, 2011.

\bibitem{cano}
Javier Cano, Csaba~D. T\'oth, Jorge Urrutia, and Giovanni Viglietta.
\newblock Edge guards for polyhedra in three-space.
\newblock {\em Computational Geometry: Theory and Applications}, 104:101859,
  2022.

\bibitem{eppstein}
David Eppstein, Elena Mumford, Bettina Speckmann, and Kevin Verbeek.
\newblock Area-universal and constrained rectangular layouts.
\newblock {\em SIAM Journal on Computing}, 41(3):537--564, 2012.

\bibitem{convex}
Branko Gr\"unbaum.
\newblock {\em Convex polytopes}.
\newblock Springer-Verlag New York, Inc., second edition, 2003.

\bibitem{barnette}
Jochen Harant.
\newblock A note on {B}arnette's conjecture.
\newblock {\em Discussiones Mathematicae Graph Theory}, 33(1):133--137, 2013.

\bibitem{kimberly}
Kimberly Kokado and Csaba~D. T\'oth.
\newblock Nonrealizable planar and spherical occlusion diagrams.
\newblock In {\em Proceedings of the 24th Japan Conference on Discrete and
  Computational Geometry, Graphs, and Games (JCDCGGG 2022)}, pages 60--61,
  2022.

\bibitem{merino}
Arturo Merino and Torsten M\"utze.
\newblock Efficient generation of rectangulations via permutation languages.
\newblock In {\em 37th International Symposium on Computational Geometry (SoCG
  2021)}, volume 189, pages 54:1--54:18, 2021.

\bibitem{art}
Joseph O'Rourke.
\newblock {\em Art gallery theorems and algorithms}.
\newblock Oxford University Press, 1987.

\bibitem{orourke}
Joseph O'Rourke.
\newblock Visibility.
\newblock In Jacob~E. Goodman, Joseph O'Rourke, and Csaba~D. T\'oth, editors,
  {\em Handbook of discrete and computational geometry}, chapter~33, pages
  875--896. CRC Press, 2017.

\bibitem{pach}
J\`anos Pach, Richard Pollack, and Emo Welzl.
\newblock Weaving patterns of lines and line segments in space.
\newblock {\em Algorithmica}, 9(6):561--571, 1993.

\bibitem{tjc}
Csaba~D. T\'oth, Jorge Urrutia, and Giovanni Viglietta.
\newblock Minimizing visible edges in polyhedra.
\newblock In {\em Proceedings of the 23rd Thailand-Japan Conference on Discrete
  and Computational Geometry, Graphs, and Games (TJCDCGGG 2020+1)}, pages
  70--71, 2021.

\bibitem{arxiv}
Csaba~D. T\'oth, Jorge Urrutia, and Giovanni Viglietta.
\newblock Minimizing visible edges in polyhedra.
\newblock {\em arXiv:2208.09702 [cs.CG]}, pages 1--19, 2022.

\bibitem{reflex}
Giovanni Viglietta.
\newblock Optimally guarding 2-reflex orthogonal polyhedra by reflex edge
  guards.
\newblock {\em Computational Geometry: Theory and Applications}, 86:101589,
  2020.

\bibitem{cccg}
Giovanni Viglietta.
\newblock A theory of spherical diagrams.
\newblock In {\em Proceedings of the 34th CanadianConference on Computational
  Geometry (CCCG 2022)}, pages 306--313, 2022.

\end{thebibliography}

\end{document}